\RequirePackage{fix-cm}
\documentclass[smallcondensed]{svjour3}       %
\usepackage{graphicx}
\usepackage{mathtools}
\usepackage{amssymb}
\usepackage[latin1]{inputenc}
\usepackage[english]{babel}
\usepackage[round,longnamesfirst]{natbib}

\newcommand{\ds}{\displaystyle}

\newcommand{\bc}{\begin{center}}
\newcommand{\ec}{\end{center}}
\newcommand{\1}{\mathbf{1}}
\newcommand{\E}{\mathbb{E}}
\newcommand{\V}{\mathbb{V}}

\renewcommand{\P}{\mathbb{P}}

\newcommand{\R}{\mathbb{R}}

\newcommand{\N}{\mathbb{N}}

\newcommand{\cB}{\mathcal{B}}
\newcommand{\cC}{\mathcal{C}}

\newcommand{\cF}{\mathcal{F}}
\newcommand{\cH}{\mathcal{H}}

\newcommand{\cU}{\mathcal{U}}
\newcommand{\cV}{\mathcal{V}}

\newcommand{\cS}{\mathcal{S}}

\newcommand{\gU}{\mathfrak{U}}
\newcommand{\gV}{\mathfrak{V}}

\newcommand{\W}{\mathbb{W}}

\newcommand{\cA}{\mathcal{A}}

\begin{document}

\title{ Global sensitivity analysis for models\\ described by stochastic differential equations}

\author{Pierre \'ETOR\'E \and Clémentine PRIEUR \and Dang Khoi PHAM \and Long LI}

\institute{P.  \'Etoré\at
              Univ. Grenoble Alpes, CNRS, Inria, Grenoble INP*, LJK, F-38000 Grenoble, France\\
*Institute of Engineering Univ. Grenoble Alpes\\
              \email{pierre.etore@univ-grenoble-alpes.fr}           
           \and
              C. Prieur \at
              Univ. Grenoble Alpes, CNRS, Inria, Grenoble INP*, LJK, F-38000 Grenoble, France\\
*Institute of Engineering Univ. Grenoble Alpes\\
              \email{clementine.prieur@univ-grenoble-alpes.fr}
              \and
              D.K. Pham \at
              Univ. Grenoble Alpes, CNRS, Inria, Grenoble INP*, LJK, F-38000 Grenoble, France\\
*Institute of Engineering Univ. Grenoble Alpes
              \and
              L. Li \at
              Univ. Grenoble Alpes, CNRS, Inria, Grenoble INP*, LJK, F-38000 Grenoble, France\\
*Institute of Engineering Univ. Grenoble Alpes
}




\maketitle

\begin{abstract}
Many mathematical models involve input parameters, which are not precisely known. Global sensitivity analysis aims to identify the parameters whose uncertainty has the largest impact on the variability of a quantity of interest. One of the statistical tools used to quantify the influence of each input variable on the quantity of interest are the Sobol' sensitivity indices. In this paper, we consider stochastic models described by stochastic differential equations (SDE). We focus the study on mean quantities, defined as the expectation with respect to the Wiener measure of a quantity of interest related to the solution of the SDE itself. Our approach is based on a Feynman-Kac representation of the quantity of interest, from which we get a parametrized partial differential equation (PDE) representation of our initial problem. We then handle the uncertainty on the parametrized PDE using polynomial chaos expansion and a stochastic Galerkin projection.

\end{abstract}

\section{Introduction}
Many mathematical models (or numerical simulators) encountered in applied sciences involve numerous poorly-known input parameters. Moreover, stochastic models are often necessary for a realistic description of the physical phenomena. In deterministic models, the model response (output) is fully determined by the values of the input parameters. A specificity of nondeterministic simulators is that the same set of parameter values will lead to an ensemble of different model responses. It is critical for practitioners to evaluate the impact of the parameter lack of knowledge onto the model response. This is the aim of sensitivity analysis. We consider in the present work the framework of global sensitivity analysis (GSA). 

In the framework of GSA for deterministic models \citep[see, e.g.,][and references therein]{saltelli2000sensitivity}, the uncertain vector of input parameters is modeled by some random vector, whose probability distribution accounts for the practitioner's belief about the input parameters uncertainty. This turns the model response into a random response, whose total variance can be split down into different parts of variance under the assumption that the input parameters are independent \citep[this is the so-called Hoeffding decomposition, see][]{van2000asymptotics}. Each part of variance corresponds to the contribution of each set of input parameters on the variance of the model response. By considering the ratio of each part of variance to the variance of the model response, we obtain a measure of importance for each set of input parameters that is called the Sobol' index or sensitivity index \cite{sobol1993sensitivity}; the most influential set of input parameters can then be identified and ranked as the set of input parameters with the largest Sobol' indices. 

Considering a similar stochastic formalism for modeling the uncertainty on the parameters of a nondeterministic model, we obtain a random response whose randomness has two sources: the randomness from the parameters, and the one due to the stochasticity inherent to the model itself. Then there are two different settings: either one is interested in the full probability distribution of the output, or we are only concerned with the output averaged over the inherent randomness of the physical system. Note that in the following, we consider that the stochasticity of the simulator can be modeled by an additional parameter, independent of the initial uncertain parameters, also considered as a random seed variable, which can be settled by the user. In other more complex frameworks, the randomness of the process is uncontrollable in the sense that it is managed by the simulator itself, and classical sensitivity analysis tools, like Monte Carlo algorithms or metamodeling, cannot be used.

In case we are only interested in the mean value relative to the intrinsic randomness of the model, a quite common procedure consists in replacing the mean quantity by the empirical mean, and then in performing usual GSA. 

If one is interested in the full probability distribution of the output, both papers \cite{iooss2009global} and \cite{marrel2012global} propose a strategy based on a joint modeling of the mean and dispersion of stochastic model outputs. From these joint models, it is possible to compute, for each uncertain parameter, the first-order Sobol' index and any Sobol' index associated to an interaction with a subset of uncertain parameters.
It is also possible to compute a total Sobol' index which contains the part of the output variance explained by the intrinsic noise of the model, by itself or in interaction with the uncertain parameters. However, such an approach does not allow to separate the effects enclosed in this total index.
In \cite{hart2017efficient}, the authors propose another point of view. They define Sobol' indices as random variables and study their statistical properties.

In the present paper, we focus our analysis on stochastic models described by stochastic differential equations (SDE). Such equations are frequently used for the simulation of complex systems, such as chemical kinetics \citep[see, e.g.,][and related works]{gillespie2003improved} or ocean models relying on stochastic parametrizations of unresolved scales \citep[see, e.g.,][and references therein]{cooper2017optimisation}, to cite only a few. 
Global sensitivity analysis for parametrized SDEs has been considered in \cite{le2015pc,jimenez2017nonintrusive}. In these two papers, the authors assume that the Wiener noise and the uncertain parameters in the parametrized SDE are independent. We will also consider this assumption. They are then interested in the full probability distribution of the solution of the parametrized SDE. Their study relies on polynomial chaos (PC) expansion \cite{wiener1938homogeneous,cameron1947orthogonal,ghanem1991stochastic,le2010spectral} to represent the uncertain parameters, leading through a Galerkin formalism to a hierarchy of stochastic differential equations governing the evolution of the modes in the PC expansion of the overall solution. The main advantage of the PC expansion of their uncertain stochastic process is that it allows a readable expression of the conditional expectations and variances.

Contrarily to the setting of both aforementioned papers, we are not considering the full probability distribution of the solution. We are rather considering mean quantities, such as the expectation with respect to the Wiener measure of a quantity of interest related to the solution itself. We focus the study on two quantities: the first one is the expectation of the exit time $\tau_D$ from a regular bounded domain $D \subset \mathbb{R}^d$, the second one is the expectation of a functional of the solution at a time $t$ on the event $t \leq \tau_D$. In this framework, we introduce a new methodology. We first use a Feynman-Kac representation of the quantity of interest \citep[see, e.g.,][]{kara,gobet-livre-num}. This leads to a parametrized partial differential equation (PDE) representation of our problem. We then handle the uncertainty on the parametrized PDE using polynomial chaos expansion and a stochastic Galerkin projection. The use of PC expansion for sensitivity analysis, or more generally uncertainty quantification, of parametrized PDE is common as it leads to analytical representations of the conditional expectations and variances, as already mentioned. 

Our paper not only proposes a practical implementation of our new methodology. It also studies its theoretical justification.
It is organised as follows: in Section~\ref{sec:notation} we fix some notation, state our problem and give some first results on the integrability of our quantities of interest. We also give a brief overview of existing methods for addressing our problem (Subsection \ref{ssec:state-art}).
In Section \ref{sec:notre-meth} we present our methodology based on Feynman-Kac formulae. Both theoretical and numerical issues are discussed. Section \ref{sec:OU} presents some numerical experiments on a toy model. Section \ref{sec:app} is an appendix gathering the proofs of several technical results, mostly related to stochastic partial differential equations.

\section{Notation, problem statement and first results}\label{sec:notation}

\subsection{First notation and assumptions}
\label{ssec:notations}

Let $\Xi\subset\R^m$, equipped with the Borel $\sigma$-field $\cB_\Xi$. Let  $\P_\xi$ be a probability measure on $(\Xi,\cB_\Xi)$. Note that by construction the random variable $\xi(z)=z$, $z\in\Xi$, defined on $(\Xi,\cB_\Xi,\P_\xi)$ has law $\P_\xi$. In the sequel we will denote in short 
$L^2(\Xi,\P_\xi)=L^2(\Xi,\cB_\Xi,\P_\xi)$. We introduce the following assumption.

\vspace{0.2cm}

\noindent {\bf Assumption 1 (A1).} Assume that the components of $\xi=(\xi_1, \ldots , \xi_m)^T$ are independent so that $\Xi$ and $\mathbb{P}_{\xi}$ have product structures
$$\Xi=\prod_{j=1}^m \Xi_j \, , \; \mathbb{P}_\xi(dz)=\prod_{j=1}^m \mathbb{P}_{\xi_j}(d z_j) \, .$$ Assume moreover that for any $j=1, \ldots , m$, the probability measure $\mathbb{P}_{\xi_j}$ is characterized by its moments. One sufficient condition for this last assertion is that the moment generating function of $\xi_j$ has positive radius of convergence. Assumption (A1) ensures the existence of an orthonormal polynomial basis of~$L^2(\Xi,\P_\xi)$.

\vspace{0.2cm}

The random variable $\xi$ will represent the uncertain parameter in the SDEs we will consider.

\vspace{0.3cm}
Now let $(C,\cC,\W)$ denote the usual $d$-dimensional Wiener space. That is to say $C$ is the space of continuous functions from $[0,\infty)$ to $\R^d$,
equipped with the $\sigma$-field $\cC$ endowed by the open sets for the metric defined in Equation $(2.4.1)$ p60 of \cite{kara}. We recall that the Wiener measure $\W$ is such that the canonical process $W(\theta)=\theta$, $\theta\in C$, is a $d$-Brownian motion on $(C,\cC,\W)$.

\vspace{0.2cm}
We now set $\Omega=C\times \Xi$, $\cF=\cC\otimes\cB_\Xi$ and $\P=\W\otimes\P_\xi$. A natural consequence of the product structure of the probability space 
$(\Omega,\cF,\P)$ is Assumption 2 below:

\vspace{0.2cm}
\noindent{\bf Assumption 2 (A2).} On $(\Omega,\cF,\P)$ the random variables $W$ and $\xi$ are independent.

\vspace{0.3cm}
Consider now $b:\R^d\times\Xi\to\R^d$ and $\sigma:\R^d\times\Xi\to\R^{d\times d}$. Denoting $|\cdot|$ both the euclidean norm on $\R^d$ and on $\R^{d\times d}$ (i.e. $|a|^2=\sum_{k,j=1}^da^2_{kj}$ for any $a\in\R^{d\times d}$) we introduce the following assumptions.

\vspace{0.2cm}
\noindent {\bf Assumption 3 (A3).} There exists a constant $0 \leq M < \infty$ such that for all $x,y \in \mathbb{R}^d $ and all $z \in \Xi$,
$$
|b(x,z)-b(y,z)|+|\sigma(x,z)-\sigma(y,z)|\leq M|x-y|$$
and
\begin{equation}
\label{eq:growth}
|b(x,z)|^2+|\sigma(x,z)|^2\leq M^2(1+|x|^2).
\end{equation}

\vspace{0.4cm}
We now consider the SDE with uncertain parameter $\xi$,
\begin{equation}
\label{eq:eds-uq}
dX_t=b(X_t,\xi)dt + \sigma(X_t,\xi)dW_t,\quad t\geq 0,\quad X_0=x.
\end{equation}

Thanks to (A3), and using Theorem 5.2.9 in \cite{kara}, we can claim that for every value of $\xi$ the SDE \eqref{eq:eds-uq} has a unique strong solution $X$. That is to say there exists $X=(X_t)_{t\geq 0}$ defined on $(\Omega,\cF,\P)$
 such that for a.e. $\omega=(\theta,z)\in\Omega=C\times\Xi$, 
we have: $\forall t\geq 0$, $\forall 1\leq i\leq d$,
$$
X^i_t(\omega)=x^i+\int_0^tb_i(X_s(\omega),\xi(z))\,ds+\sum_{j=1}^d\sigma_{ij}(X_s(\omega),\xi(z))\,dW^j_s(\theta).$$
In fact, under (A3), such a process has all its moments as we state in the following lemma, whose proof is postponed to the Appendix.

\begin{lemma}
\label{lem:X-L2}
Assume (A2) and (A3). Let $X$ be the solution of \eqref{eq:eds-uq} and $q\in\N^*$. For any $t\geq 0$ we have $\E|X_t|^q<\infty$ (in other words $X_t\in L^q(\Omega,\P)$).
\end{lemma}

\begin{remark}
Assumption 3, in order to get the existence of the solution $X$ and its square integrability, could be weakened. But it provides a reasonable setting to study the partial differential equation associated to our original problem (see Section~\ref{sec:notre-meth}).
\end{remark}

\begin{remark}
\label{rem:esp-xi}
Let $F:C\to \R$ measurable such that $F(X)$ belongs to $L^1(\Omega,\P)$. Note that thanks to (A2),
$$
\E[\,F(X)\,|\,\xi\,]=\int_CF(X(\theta,\xi))\W(d\theta),$$
that is to say $\E[\,\cdot\,|\,\xi\,]$ can be seen as the average with respect to the Brownian motion $W$ driving \eqref{eq:eds-uq}.
\end{remark}

\subsection{Problem statement}

As already mentioned in the introduction, the $d$-Brownian motion in our study models the physical system randomness. The quantities we are interested in are then mean values, with respect to the $d$-Brownian motion, of outputs related to the SDE. These quantities depend on uncertain parameters, and we aim at determining which parameters are influential on these quantities, by performing a global sensitivity analysis. This section aims at defining properly two quantities we are interested in (see Equations \eqref{firstquantity} and \eqref{eq:def-quant-int2}). 

\medskip

We first need to introduce a set of assumptions and notation, as far as preliminary results.

Let $\sigma$ denote the diffusion term in Equation \eqref{eq:eds-uq}, we then define $a=\sigma\sigma^T$ and introduce the following uniform ellipticity assumption:

\vspace{0.2cm}
\noindent {\bf Assumption 4 (A4).} There exists $\lambda>0$ such that
$$
\forall x,y\in\R^d,\;\forall z\in\Xi,\quad y^Ta(x,z)y\geq \lambda |y|^2.$$

\vspace{0.4cm}
Let $D$ be an open bounded subset of $\R^d$. Recall that $X$ is solution of \eqref{eq:eds-uq}. We define the first exit time of the process $X$ from $D$ as:
$$
\tau_D=\inf\{ t\geq 0: \,X_t\notin D\}.$$

We have the following result (Lemma \ref{lem:tau-L1}) whose proof is postponed to the Appendix.

\begin{lemma}
\label{lem:tau-L1}
Assume (A2)-(A4). For any starting point $x\in D$ of \eqref{eq:eds-uq} we have $\tau_D\in L^1(\Omega,\P)$.
\end{lemma}

\vspace{0.2cm}

In the sequel the starting point $x\in\R^d$ of $X$ solution of \eqref{eq:eds-uq} may vary. We classically denote by $\E^x(\cdot|\xi)$ the expectation with respect to $W$ computed under the initial condition $X_0=x$.

\medskip

The first quantity we are interested in is the averaged (with respect to $W$) exit time of the process $X$ from the domain $D$ (we assume $x\in \bar{D}$, the closure of $D$). It is defined as:

\begin{equation}\label{firstquantity}
\gU(x,\xi)=\E^x[\,\tau_D\,|\,\xi\,].
\end{equation}

\medskip

To ensure that $\gU(x,\xi)$ is in $L^2(\Xi,\mathbb{P}_{\xi})$, we add a regularity assumption on the domain $D$:

\vspace{0.2cm}
\noindent {\bf Assumption 5 (A5).} The boundary $\partial D$ is of class $C^{2,\alpha}$, $\alpha=1$ \citep[see][p94 for a definition]{trudinger}.

\medskip

Then, it is possible to prove the following lemma:

\begin{lemma}
\label{coro:gU-L2}
Assume (A2)-(A5). For any $x\in D$ and almost every value of $\xi$ we have $|\gU(x,\xi)|\leq C$ where $C$ is a finite constant. 
\end{lemma}

\begin{proof}
Note that Lemma \ref{coro:gU-L2} is a corollary of Theorem~\ref{thm:feyn-U} stated in Section \ref{sec:notre-meth}. Its proof is postponed to the Appendix. 
\end{proof}

\begin{remark}
Even if the result of Theorem~\ref{thm:feyn-U} is required for the proof of Lemma~\ref{coro:gU-L2}, we decided to postpone its statement to Section \ref{sec:notre-meth}, devoted to the introduction of our new methodology, based on Feynman-Kac representations of the mean quantities we are interested in. Theorem~\ref{thm:feyn-U} provides indeed a clear interpretation of $\gU(x,\xi)$ as the solution of an elliptic stochastic partial differential equation.
\end{remark}

\medskip

Let now $0\leq  t<\infty$. Let $f$ be a function from $\bar{D}$ to $\R$. We introduce the following set of assumptions on $f$:

\vspace{0.2cm}
\noindent {\bf Assumption 6 (A6).} The function $f$ is of class $C^2(\bar{D})$ and satisfies the compatibility condition
\begin{equation}
\label{eq:compatibility}
 f\equiv 0
\quad\text{and}\quad \frac 1 2\sum_{k,j}^da_{kj}\partial^2_{x_kx_j}f+\sum_{j=1}^db_j\partial_{x_j}f\equiv 0
\text{ on }\partial D.
\end{equation}
Besides, for all $x,y\in \bar{D}$,
$$
|f(x)-f(y)|+\sum_{k=1}^d|\partial_{x_k}f(x)-\partial_{x_k}f(y)|+\sum_{k,j=1}^d|\partial^2_{x_jx_k}f(x)-\partial^2_{x_jx_k}f(y)|\leq M|x-y|
$$
and 
\begin{equation}
\label{eq:controle-f}
|f(x)|\leq M(1+|x|^\mu),
\end{equation}
with  constants $0\leq M,\mu<\infty$.

\vspace{0.3cm}
From Lemma \ref{lem:X-L2} we immediately get Lemma \ref{lem:fX-Lq} just below.

\begin{lemma}
\label{lem:fX-Lq}
Assume (A2)-(A3) and \eqref{eq:controle-f}. Then for any $q\in\N^*$ we have $f(X_t)\in L^q(\Omega,\P)$, for any~$t\geq 0$.
\end{lemma}

It is then possible to define the second quantity we are interested in. The quantity $\gV(t,x,\xi)$  is defined as the expectation (with respect to $W$) of $f(X_t)$, computed on the event $\{t\leq \tau_D\}$:

\begin{equation}
\label{eq:def-quant-int2}
\gV(t,x,\xi)=\E^x[\,f(X_t)\1_{t\leq \tau_D}\,|\,\xi\,].
\end{equation}

\medskip

\begin{remark}
Lemma~\ref{lem:fX-Lq} ensures that, under (A2)-(A3) and (A6), the quantity $\gV(t,x,\xi)$ is in $L^2(\Xi,\mathbb{P}_{\xi})$. 
\end{remark}

In the sequel, in order to lighten notation, we will sometimes drop any reference to $x$ or $t$ in the quantities of interest 
$\gU$ or $\gV$.

\medskip

Recall that we aim at determining which parameters, or set of parameters, are influential for $\gU$, {\it resp.} $\gV$. We propose to compute Sobol' sensitivity indices, with respect to the components of~$\xi=(\xi_1,\ldots,\xi_m)^T$ assumed to be independent. These indices are defined \citep[see, e.g.,][]{sobol1993sensitivity} as:

\begin{equation}
\label{eq:sobol-U}
S_I(\gU)=\frac{\V[\E(\gU(\xi)|\xi_I)]}{\V(\gU(\xi))},\quad\forall I\subset \{1,\ldots,m\},\quad\xi_I=\{\xi_\ell, \,\ell\in I\}
\end{equation}
and
\begin{equation}
\label{eq:sobol-V}
S_I(\gV)=\frac{\V[\E(\gV(\xi)|\xi_I)]}{\V(\gV(\xi))},\quad\forall I\subset \{1,\ldots,m\},\quad\xi_I=\{\xi_\ell, \,\ell\in I\}\, .
\end{equation}

\begin{remark}
In the following, we assume that $x\in D$ and $t>0$ to ensure $\V(\gU(\xi))\neq 0$ and $\V(\gV(\xi))\neq 0$.
\end{remark}

\begin{remark}\label{rem:A6}
Assumption (A6) on $f$ is a technical assumption which may appear restrictive. In practice, if $f$ does not satisfy (A6), we approximate it by some $f_\varepsilon$ satisfying (A6) and propose a bound to control the approximation error. E.g., for $f \equiv 1$ on $\bar{D}$, and for any $\varepsilon >0$, it is possible via convolution arguments to construct a function $f_\varepsilon \in C^2( \bar{D})$ satisfying (A6), such that $0 \leq f_{\varepsilon} \leq 1$ and $f_\varepsilon  \equiv 1$ on $K_\varepsilon$, with $K_\varepsilon$ a compact subset of $D$ with $|D \setminus K_\varepsilon| \leq \varepsilon$ (note that the existence of $K_\varepsilon$ is ensured by Assumption~(A5) on the regularity of the boundary). Then we have:
\begin{equation}\label{approxrem}
\left| \E^x[\,f(X_t)\1_{t\leq \tau_D}\,|\,\xi\,]-\E^x[\,f_\varepsilon(X_t)\1_{t\leq \tau_D}\,|\,\xi\,]\right| \leq  \, \mathbb{P}^x \left(X_t \in D\setminus K_\varepsilon \, | \, \xi \right) \, .
\end{equation}
But, modifying $\sigma$ and $b$ outside $D$ so that they are bounded and Lipschitz on $\R^d$ (with constants uniform in $\xi$), and using Theorems 4.5 and 5.1 in \cite{friedman2}, we get the following estimate for the transition probability density of $(X_t)_{t \geq 0}$, under assumptions (A3), (A4):
\begin{equation}
\label{eq:aronson}
p(t,x,y,\xi)\leq \frac{M_T}{t^{d/2}}\exp\Big( -\frac{|y-x|^2}{M_Tt}  \Big),\quad \forall x,y \in D
\end{equation}
with $M_T$ depending on $T$, $M$, $\lambda$ and $|D|$ (but not on $\xi$). 
We thus conclude that the right hand side in~\eqref{approxrem} behaves as $O \left(\varepsilon\right)$.
\end{remark}

\subsection{State of the art}
\label{ssec:state-art}

In this section, we first recall the crude Monte Carlo procedure for the estimation of $S_I(\gU)$ and $S_I(\gV)$. We then present the approach introduced recently in \cite{jimenez2017nonintrusive,le2015pc} for the estimation of the $S_I(\gV_i)$, with $\gV_i(t,x,\xi)=\mathbb{E}^x[X_t^i|\xi]$, for $1 \leq i \leq d$.

\subsubsection{Crude Monte Carlo estimation procedure}
\label{sssec:MC}

The crude Monte Carlo estimation procedure is probably the most intuitive. However, its main issue is its computational cost.
Let us detail below what we mean by crude Monte Carlo estimation procedure.

\medskip

Let $\cF_\gU:C\times\Xi\to\R$ ({\it resp.} $\cF_\gV:C\times\Xi\to\R$) be the application such that
$
\cF_\gU(W,\xi)=\tau_D$ ({\it resp.} $
\cF_\gV(W,\xi)=f(X_t)\1_{t\leq\tau_D}$).
Note that the deterministic maps $\cF_\gU$ and $\cF_\gV$ exist because for fixed $\xi$ one passes from a path of $W$ to a path of $X$ in a deterministic way, thanks to the Yamada-Watanabe causality principle \citep[see][]{kara}.

In practice our Monte Carlo procedures will involve some Euler scheme to approximate the paths of $X$, so that we will in fact work with approximations $\widehat{\cF}_\gU$ and  $\widehat{\cF}_\gV$ of $\cF_\gU$ and $\cF_\gV$ (e.g., in Section \ref{sec:OU}).

In the sequel we focus on $\gU(\xi)$, as $\gV(\xi)$ can be treated in a similar manner. The quantity $\gU(\xi)$  is first approximated by
$$
\E[\,\widehat{\cF}_\gU(W,\xi)\,|\,\xi\,].$$
Second, let $M\in \N^*$ ($M$ is supposed to be large) and $W^{(k)}$, $k=1,\ldots,M$,
independent samplings of the Brownian motion~$W$. We approximate $\E[\,\widehat{\cF}_\gU(W,\xi)\,|\,\xi\,]$ by the Monte Carlo mean
$$
\overline{E}^M\big[ \widehat{\cF}_\gU\big](\xi):=\frac 1 M \sum_{k=1}^M \widehat{\cF}_\gU(W^{(k)},\xi).$$

We then perform a sensitivity analysis of  $\overline{E}^M\big[ \widehat{\cF}_\gU\big](\xi)$ with respect to the components of $\xi=(\xi_1,\ldots,\xi_m)^T$.
Sobol' indices can be estimated with a classical {\it pick freeze} procedure \citep[see, e.g.,][]{sobol1993sensitivity,sobol2001global}. Let $I\subset\{1,\ldots,m\}$. Let $N\in \N^*$ ($N$ is supposed to be large).  Let $\xi^{A,(l)}$,$\xi^{B,(l)}$, $l=1,\ldots,N$ independent samplings of $\xi$. We then construct samples $\xi^{I,(l)}$, $l=1,\ldots,N$ in the following manner:
$$
\forall 1\leq l\leq N, \forall 1\leq \ell \leq m, \quad \xi^{I,(l)}_\ell=\left\{
\begin{array}{ll}
\xi^{B,(l)}_\ell&\text{if } \ell\in I\\
\xi^{A,(l)}_\ell&\text{otherwise}.\\
\end{array}
\right.
$$
Then the Sobol' index $S_I(\gU)$ is approximated by
$$
S_I(\overline{E}^M\big[ \widehat{\cF}_\gU\big](\xi))=\frac{\V\big[\E\big(\overline{E}^M\big[ \widehat{\cF}_\gU\big](\xi)\,|\,\xi_I\big)\big]}{\V\big[\overline{E}^M\big[ \widehat{\cF}_\gU\big](\xi)\big]}=\frac{\V\big[\overline{E}^M\big[ \widehat{\cF}_\gU\big](\xi_I)\big]}{\V\big[\overline{E}^M\big[ \widehat{\cF}_\gU\big](\xi)\big]}$$
which itself is estimated by
\begin{equation}
\label{eq:MC}
\frac{\dfrac 1 N \sum_{l=1}^N  \overline{E}^M\big[ \widehat{\cF}_\gU\big](\xi^{I,(l)}) \times \overline{E}^M\big[ \widehat{\cF}_\gU\big](\xi^{B,(l)})  -\Big(\overline{C}\Big)^2}{\dfrac 1 N \sum_{l=1}^N \dfrac{\big\{ \overline{E}^M\big[ \widehat{\cF}_\gU\big](\xi^{I,(l)}) \big\}^2+  \big\{\overline{E}^M\big[ \widehat{\cF}_\gU\big](\xi^{B,(l)}) \big\}^2}{2}
-\Big(\overline{C}\Big)^2}
\end{equation}
with 
$$\overline{C}=\dfrac 1 N \sum_{l=1}^N \dfrac{ \overline{E}^M\big[ \widehat{\cF}_\gU\big](\xi^{I,(l)})  +  \overline{E}^M\big[ \widehat{\cF}_\gU\big](\xi^{B,(l)}) }{2}\cdot$$

This estimator has been introduced first in \cite{monod2006uncertainty} and its asymptotic properties have been studied in \cite{janon2014asymptotic}. The main advantage of {\it pick freeze} estimators is that it only requires the square integrability for the model.
However, the number of model evaluations it requires can be huge. It is equal to $MN(m+1)$, with $m$ the number of uncertain parameters.

\subsubsection{Hybrid Galerkin-Monte Carlo procedure}
\label{sssec:lemaitre}

In the setting of parametrized SDE, we can exploit regularity properties of the underlying model for proposing alternatives to the crude Monte Carlo procedure. In this section, we present the procedure introduced in \cite{le2015pc,jimenez2017nonintrusive}. It is based on a polynomial chaos analysis with stochastic expansion coefficients. In that paper, the uncertainty due to the Wiener noise and the one due to the uncertain parameters $\xi$ are handled at the same level for sensitivity analysis.
Recall that the components of $\xi=(\xi_1, \ldots , \xi_m)^T$ are assumed to be independent so that $\Xi$ and $\mathbb{P}_{\xi}$ have product structures
$$\Xi=\prod_{j=1}^m \Xi_j \, , \; \mathbb{P}_\xi(dz)=\prod_{j=1}^m \mathbb{P}_{\xi_j}(d z_j) \, .$$
We endow ${L}^2(\Xi,\mathbb{P}_{\xi})$ with the inner product and associated norm denoted $<\cdot,\cdot>$ and $\| \cdot \|_2$ respectively
$$\forall \, U,V \in L^2(\Xi,\mathbb{P}_{\xi}) \, , \; <U,V>=\int_{\Xi}U(z)V(z)\mathbb{P}_{\xi}(dz) \, ;$$
$$\|U\|_2=<U,U>^{1/2}< \infty \Longleftrightarrow U \in L^2(\Xi,\mathbb{P}_{\xi}) \, .$$
Then, introducing an orthonormal basis $\{\Psi_q(\cdot)\, , \; q \in \mathbb{N}\}$ for $L^2(\Xi,\mathbb{P}_{\xi})$ (e.g., any tensorized basis), any second-order random variable $U(\xi)$ can be expanded as 
$$U(\xi)=\sum_{q \in \mathbb{N}} [U_{q}]\Psi_{q}(\xi) \, .$$
The authors in \cite{le2015pc,jimenez2017nonintrusive} consider then the following tensor representation: for a.e. $\omega=(\theta,z)\in\Omega=C\times\Xi$,

\begin{equation}
\label{eq:decomp-X}
X_t^i (\theta,z)=\sum_{q \in \mathbb{N}}[{(X_t^i)}_q](\theta)\Psi_q(z) \, , \; i=1 , \ldots , d.
\end{equation}
Then, they consider a stochastic Galerkin projection (see, e.g., \cite{ghanem1991stochastic}) to derive equations that enable the determination of the stochastic processes $[{(X_t^i)}_q](\theta)$, $i=1, \ldots , d$.
E.g., let $i \in \{1, \ldots , d\}$, we get\\\\
\noindent $d[{(X_t^i)}_q](\theta)  =$\\
\noindent $<b(\sum_{r \in \mathbb{N}}[{(X_t^i)}_r](\theta)\Psi_r(\cdot) ,\cdot )dt,\Psi_q(\cdot)>+\sigma(\sum_{r \in \mathbb{N}}[{(X_t^i)}_r](\theta)\Psi_r(\cdot) ,\cdot )dW_t(\theta),\Psi_q(\cdot)>$.\\

From Assumption (A2), we know that $\xi$ and $W$ are independent thus\\\\
\noindent $d[{(X_t^i)}_q](\theta)  =$\\
\noindent $<b(\sum_{r \in \mathbb{N}}[{(X_t^i)}_r](\theta)\Psi_r(\cdot) ,\cdot ),\Psi_q(\cdot)>dt+\sigma(\sum_{r \in \mathbb{N}}[{(X_t^i)}_r](\theta)\Psi_r(\cdot) ,\cdot ),\Psi_q(\cdot)>dW_t(\theta)$.\\

From computational purposes, the authors truncate this infinite sequence of coupled problems to an order $P$: $0 \leq q \leq P$. 
They thus get a system of $P+1$ coupled stochastic differential equations, with initial conditions obtained by projecting the initial data on the stochastic basis
$[{(X_0^i)}_q](\theta)=<X_0,\Psi_q>$. The system is then solved using standard Monte Carlo simulation, introducing a time scheme and generating trajectories for the $d$-Brownian motion.

Consider now the quantities of interest analyzed in \cite{jimenez2017nonintrusive,le2015pc}:
$$
\forall 1\leq i\leq d,\quad \gV_i(t,x,\xi)=\E^x[X^i_t\,|\,\xi],$$
denoted in short $\gV_i(\xi)$, $1\leq i\leq d$, and recall \eqref{eq:decomp-X}.

Then, due to the orthonormality of the stochastic basis $\{\Psi_q \, , \; q \in \mathbb{N}\}$, we get
\begin{equation}\label{approx:var}
\mathbb{V}(\gV_i(\xi))\approx  \mathbb{V}\{[(X_t^i)_0]\}+\sum_{q=1}^P \mathbb{E} \{ [(X_t^i)_q]^2\}
\end{equation}
and
\begin{equation}\label{approx:num}
\mathbb{V}[\mathbb{E}(\gV_i(\xi)|\xi)]\approx \sum_{q=1}^P \{ \mathbb{E}[(X_t^i)_q]\}^2
\end{equation}
for the estimation of Sobol' indices defined by \eqref{eq:sobol-V} with $I=\{1, \ldots , m\}$.

The application of such a procedure requires technical assumptions on the initial parametrized SDE, which we do not develop here for sake of concision.

\section{A new methodology based on Feynman-Kac representation formulas}
\label{sec:notre-meth}

This section is devoted to the presentation of the new methodology we propose for performing sensitivity analysis of parametrized stochastic differential equations. Let us recall that we consider in our approach that our model is a stochastic model depending on unknown parameters $\xi=(\xi_1, \ldots , \xi_m)^T$, the stochasticity being induced by the $d$-Brownian motion. We then are only interested by the mean value of a functional of the model output, where the mean is taken with respect to the $d$-Brownian motion which drives the inherent randomness of the model. It is different in nature from the point of view adopted in \cite{le2015pc,jimenez2017nonintrusive}, in which the authors are interested in the sensitivity of the model output with respect to both the uncertain parameters and the noise inherent to the system.
Our approach is based on Feynman-Kac representation formulas which establish a link between parabolic or elliptic partial differential equations and stochastic processes.
More precisely, Feynman-Kac formulas allow interpreting the quantities $\gU$ and~$\gV$ as the solutions of some partial differential equations. 
The literature proposes a broad range of methods for the sensitivity analysis of parametrized partial differential equations \citep[see, e.g.,][and references therein]{nouy2017low}. Hereafter, for each of both quantities of interest under study in this paper, we focus on a method based on a stochastic Galerkin polynomial chaos approximation of the solution of the parametrized PDE we obtain from the Feynman-Kac representation.

In Subsection \ref{ssec:feyn-kac} below, we state in Theorems \ref{thm:feyn-U} and \ref{thm:feyn-V} the Feynman-Kac representations of our quantities of interest $\gU$ and $\gV$. Then in Subsection \ref{ssec:ell} (resp. Subsection~\ref{ssec:para}), we introduce our methodology based on stochastic Galerkin approximation for the estimation of Sobol' indices associated to $\gU$ (resp. $\gV$). 
\medskip

\subsection{Link between $\gU$ and $\gV$ and some stochastic partial differential equations}
\label{ssec:feyn-kac}

Theorem \ref{thm:feyn-U} below provides an interpretation of $\gU$ as the solution of an elliptic problem. 

\begin{theorem}
\label{thm:feyn-U}
Let us consider the following elliptic stochastic partial differential equation
\begin{equation}
\label{eq:Pb-ell}
\left\{\begin{array}{rlll}
\frac 1 2 \sum_{k,j=1}^da_{kj}(x,\xi)\partial^2_{x_kx_j}u(x,\xi)+\sum_{k=1}^db_k(x,\xi)\partial_{x_k}u(x,\xi)&=&-1&\forall x\in D\\
\\
u(x,\xi)&=&0&\forall x\in\partial D.\\
\end{array}
\right.
\end{equation}
\begin{enumerate}
\item[i)]
Assume (A3)-(A5). Then, for almost every value of $\xi$, there exists a unique solution to \eqref{eq:Pb-ell} in 
$C^2(\overline{D})$.
\item[ii)]
If moreover (A2) is satisfied, then for {\it a.e.} $\xi$, the quantity $\gU(x,\xi)$ is solution of the partial differential equation defined by \eqref{eq:Pb-ell}. 
\end{enumerate}
\end{theorem}


\begin{proof}[Proof of Theorem \ref{thm:feyn-U}]
Point $i)$ is a direct consequence of Theorem 6.14 in \cite{trudinger}.
Taking into account Remark \ref{rem:esp-xi} and Lemma \ref{lem:tau-L1}, $ii)$ follows simply from the elliptic Feynman-Kac formula given in Theorem 5.7.2 in \cite{kara}.
\end{proof}

We now turn to the interpretation of $\gV$. 
\vspace{0.2cm}
\begin{theorem}
\label{thm:feyn-V}
Let us consider the following parabolic stochastic partial differential equation
\begin{equation}
\label{eq:Pb-para}
\left\{\begin{array}{rcl}
\partial_t v(t,x,\xi)&=&\frac 1 2 \sum_{k,j=1}^da_{kj}(x,\xi)\partial^2_{x_kx_j}v(t,x,\xi)+\sum_{k=1}^db_k(x,\xi)\partial_{x_k}v(t,x,\xi),\\\\
&&\;\;\;\;\;\;\;\;\;\;\;\;\;\;\;\;\;\;\;\;\;\;\;\;\;\;\;\;\;\;\;\;\;\;\;\;\;\;\;\;\;\;\;\;\;\;\;\;\;\;\;\;\;\;\;\;\;\;\;\;\forall (t,x)\in(0,T]\times D\\
\\
v(t,x,\xi)&=&0, \;\forall (t,x)\in (0,T]\times\partial D\\
\\
v(0,x,\xi)&=&f(x),\;\forall x\in\bar{D}.\\
\end{array}
\right.
\end{equation}

i) Assume (A3)-(A6). Then for almost every value of $\xi$, there exists a unique solution to~\eqref{eq:Pb-para} in 
$C^{1,2}([0,T]\times \bar{D})$.

\vspace{0.2cm}
ii) Assume in addition (A2). Then the quantity $\gV(t,x,\xi)$,  $0<t\leq T$, is given by the solution at time~$t$ of the parabolic stochastic partial differential equation \eqref{eq:Pb-para}.
\end{theorem}

\begin{proof}[Proof of Theorem \ref{thm:feyn-V}]
 As $f$ satisfies the compatibility condition 
\eqref{eq:compatibility}
one may extend it for any $\xi$ in a function $\bar{f}$ of class $C^{1,2}([0,T]\times\bar{D})$, with uniformly Lipschitz derivatives (to order $1$ in time and up to order $2$ in space) and satisfying $\bar{f}\equiv 0$ on $(0,T]\times\partial D$,   and
$\bar{f}\equiv f$ on $\{0\}\times\bar{D}$
and
$\frac 1 2\sum_{k,j}^da_{kj}\partial^2_{x_kx_j}\bar{f}+\sum_{j=1}^db_j\partial_{x_j}\bar{f}\equiv~0$
 on $\{0\}\times\partial D$ (e.g., let us consider $\bar{f}(t,x)=f(x) \; \forall \, (t,x) \in [0,T]\times\bar{D}$). 
Point i) then follows from Theorem 5.14 in \cite{lieberman}.
Taking into account Remark \ref{rem:esp-xi}, the result ii) follows simply from the parabolic Feynman-Kac formula given in Theorem 4.4.5 in \cite{gobet-livre-num}. Note however that this Feynman-Kac formula is given for parabolic problems in backward form and with terminal condition instead of initial one, as the authors deal with the general case of time-inhomogeneous coefficients. But in the case of time-homogeneous coefficients, one may use a reverting time argument, and the time-homogeneous Markov property of $X$, in order to get the Feynman-Kac formula for the parabolic problem in the forward form and with initial condition.
\end{proof}

At this stage of the paper, we are concerned with {\it classical solutions} of stochastic PDEs, in the sense of Definition \ref{classic} below.
In Subsections \ref{ssec:ell} and \ref{ssec:para}, we will introduce the notion of {\it weak solutions at the stochastic level}.

\begin{definition}\label{classic}
A function $u:\bar{D}\times \Xi\to\R$ (resp. $v:[0,T]\times\R^d \times \Xi\to\R$) is a classical solution of~\eqref{eq:Pb-ell}
(resp. \eqref{eq:Pb-para}) if for a.e. value of $\xi$ the function $x\mapsto u(x,\xi)$ (resp. $(t,x)\mapsto v(t,x,\xi)$) is a solution to~\eqref{eq:Pb-ell} (resp. \eqref{eq:Pb-para}) in $C^2(\overline{D})$ (resp. $C^{1,2}([0,T]\times\bar{D})$).
\end{definition}

\begin{corollary}
\label{cor:sol-classique}
Assume (A3)-(A5) (resp. (A3)-(A6)). There exists a unique classical solution of the stochastic PDE \eqref{eq:Pb-ell} (resp. \eqref{eq:Pb-para}). If in addition $(A2)$ is satisfied, this classical solution is provided by 
$\gU(x,\xi)$ (resp. $\gV(t,x,\xi)$).
\end{corollary}

From the Feynman-Kac representations of our quantities of interest stated in Corollary \ref{cor:sol-classique}, we propose in Subsections \ref{ssec:ell}
and \ref{ssec:para} to apply a methodology based on stochastic Galerkin polynomial chaos approximations to approximate Sobol' indices. 
 For the clarity of exposure we have found convenient to present  first the elliptic case (Subsection \ref{ssec:ell}), and then to turn to the parabolic one which is technically more cumbersome (Subsection~\ref{ssec:para}).

\subsection{Approximation of $S_I(\gU)$ using the elliptic stochastic PDE}
\label{ssec:ell}

We want now to construct a discretized approximation of $\gU(\xi)$. So far (Corollary~\ref{cor:sol-classique}) the function $\gU$
is seen as a classical solution of \eqref{eq:Pb-ell} but it can be shown to be a {\it weak solution at the stochastic level} of some equivalent PDE in divergence form, living in the space $H^1_0(D)\otimes L^2(\Xi,\P_\xi)$ (this will be addressed in Theorem 
\ref{thm:sumup-ell} below; the definition of $H^1_0(D)$ will be recalled in Subsubsection \ref{sssec:not-sobolev}). 
Therefore it is possible to perform a Galerkin projection on some finite dimensional subspace $V^N\otimes \cS^P$
where $\cS^P=\mathrm{Span}(\{\Psi_q \}_{q=0}^P)$ with $\{\Psi_q\}_{q=0}^P$ introduced in Subsection \ref{sssec:lemaitre}
 and $V^N=\mathrm{Span}(\{\phi^N_i\}_{i=1}^{N-1})$ with $\{\phi^N_i\}_{i=1}^{N-1}$ some finite element basis of functions in $H^1_0(D)$ (see Subsubsection \ref{sssec:not-sobolev} for details). That is to say $\gU(x,\xi)$ will be discretized both with respect to the space variable $x$ and the uncertain parameter $\xi$. Concerning weak solutions and Galerkin projection we are inspired by the framework in \cite{nouy1}, \cite{nouy2017low}. 
In this article we provide a rigorous proof of the convergence of our Galerkin approximation toward the weak solution when the discretization parameters in $x$ and $\xi$ simultaneously converge (see in particular the proof of Lemma \ref{lem:sol-faible-ell}, postponed to the Appendix).
 Once we have got an approximation (also called reduced model or metamodel)~$\gU^{N,P}$ of the weak solution, therefore of $\gU(\xi)$, we deduce an analytical formula for approximating Sobol' indices $S_I(\gU)$, $I \subset \{1, \ldots , m\}$. The arguments to derive such an analytical formula are similar to the ones used to derive \eqref{approx:var} and \eqref{approx:num} in Subsection \ref{sssec:lemaitre}, and mainly rely on the orthonormality of the basis~$\{\Psi_q \}_{q=0}^P$.

\vspace{0.3cm}

The subsection is organised ad follows: 
in Subsubsection \ref{sssec:not-sobolev} we introduce some notation on the functional spaces we will use both in this subsection and in Subsection \ref{ssec:para} about the parabolic case.
In Subsubsection \ref{sssec:div} we reformulate the elliptic stochastic PDE~\eqref{eq:Pb-ell} in divergence form. This is needed in order to  introduce in Subsubsection~\ref{sssec:weak} the definition of weak solutions at the stochastic level of elliptic stochastic PDEs, and of their Galerkin approximation. The aforementioned Theorem \ref{thm:sumup-ell} is stated in Subsubsection~\ref{sssec:goal} and establishes the convergence of the Galerkin approximation $\gU^{N,P}$ toward the quantity of interest $\gU$. In Subsubsection \ref{sssec:sob} it is explained how to use  $\gU^{N,P}$ in order to compute approximated Sobol' indices of $\gU$.

\subsubsection{Functional spaces notation}
\label{sssec:not-sobolev}

We denote by $H^1_0(D)$ the usual Sobolev space of functions in $H^1(D)$ vanishing on~$\partial D$ (note that in the sequel the derivatives may be understood in the weak sense). 
The space $H^1_0(D)$ is equipped with the norm $u\mapsto ||u||_{H^1(D)}=\int_D u^2+\int_D|\nabla u|^2$.
We denote $H^{-1}(D)$ the topological dual space of $H^1_0(D)$. 
Recall that we have the Gelfand triple $H^1_0(D)\hookrightarrow L^2(D)\hookrightarrow H^{-1}(D)$
(embeddings are continuous and dense). 

 We then introduce the tensor Hilbert spaces $\cH=L^2(D)\otimes L^2(\Xi,\P_\xi)$ 
and $\cV=H^1_0(D)\otimes L^2(\Xi,\P_\xi)$. The space $\cH$ 
is isomorphic to 
$L^2(\Xi,L^2(D),\P_\xi)$ and $L^2(D,L^2(\Xi,\P_\xi))$.
The space $\cV$ is isomorphic to 
$L^2(\Xi,H^1_0(D),\P_\xi)$ and 

$H^1_0(D,L^2(\Xi,\P_\xi))$. 
The space $\cH$ is equipped with the scalar product $\langle \cdot,\cdot\rangle_\cH$  
defined by
 $$\forall v,w\in\cH,\quad\langle v,w\rangle_\cH=\E_\xi\big(\int_D vw \big),$$
and with the norm $u\mapsto||u||_\cH:=\sqrt{\langle u,u\rangle_\cH}$.
The space $\cV$ is equipped 
  with the norm $||\cdot||_\cV$ defined as:
$$
u\mapsto \E_\xi||u(\cdot,\xi)||^2_{H^1(D)}=:||u||^2_\cV.$$
 
Denoting $\cV'$ the topological dual of $\cV$  we have the new Gelfand triple 
$\cV\hookrightarrow\cH\hookrightarrow\cV'$. Note that $\cV'=H^{-1}(D)\otimes L^2(\P_\xi)$. Note that $\langle \cdot,\cdot\rangle_\cH$ will also denote the  extension by continuity  
of the previous scalar product
to the dual pairing $\cV'\times\cV$.

\vspace{0.6cm}
\noindent
We introduce the finite dimensional approximation space $\cV^{N,P}\subset \cV$ defined as the tensor space 
\begin{equation}
\label{eq:cVNP}
\cV^{N,P}=V^N\otimes \cS^P.
\end{equation}
Recall that the space $\cS^P$ in \eqref{eq:cVNP} is defined as $\mathrm{Span}(\{\Psi_q \}_{q=0}^P)$ with $\{\Psi_q\}_{q=0}^P$ introduced in Subsection \ref{sssec:lemaitre}. Recall also that
the space $V^N$ is defined as $\mathrm{Span}(\{\phi^N_i\}_{i=1}^{N-1})$ with $\{\phi^N_i\}_{i=1}^{N-1}$ some finite element basis of functions in $H^1_0(D)$. More precisely, we choose $\phi^N_i$, $N \geq 2$, $1 \leq N-1$ such that $V_2 \subset \ldots \subset V_{N-1}\subset V_{N}$. This can be done for example by using linear $B$-spline basis functions.
Using interpolation properties, we get that $\cup_{N\geq 2} V_N$ is dense in $H^1_0(D)$. 

\medskip

We then have 
$\cV^{N,P}\subset\cV^{N+1,P+1}$  for any $N,P$, and $\bigcup_{N,P}\cV^{N,P}$ is dense in $\cV$.

\subsubsection{Elliptic PDE in divergence form}\label{sssec:div}

 For any vector field $U\in C^1(D;\R^d)$ we denote $\nabla\cdot U=\sum_{k=1}^d\partial_{x_k}U^k$  the divergence of $U$ and, for any $U\in C^1(D;\R)$ we denote $\nabla U=(\partial_{x_1}U,\ldots,\partial_{x_d}U)^T$ the gradient of $U$.
 
 \vspace{0.2cm}
 
 Assume that for a.e. value of $\xi$ the function $a(\cdot,\xi)$ is of class $C^1(D;\R^{d\times d})$ and define the coefficients $\tilde{a}$, $\tilde{b}$, the vector field $\partial \tilde{a}$, and the function $\tilde{f}$ on $D \times \Xi$ by
 \begin{equation}
 \label{eq:deftilde}
 \begin{array}{c}
\tilde{a}=a/2,\quad\quad(\partial \tilde{a})_j =\nabla\cdot(\tilde{a}_{1j},\ldots,\tilde{a}_{dj})^T,\;\;\forall \, 1\leq j\leq d\\
\\
\tilde{b}=\partial\tilde{a}-b\quad\text{ and }\quad \tilde{f}=1
\end{array} 
 \end{equation}
 
It is  then clear that \eqref{eq:Pb-ell} is equivalent to
\begin{equation}
\label{eq:elldiv}
\left\{\begin{array}{rlll}
-\nabla\cdot\big(\tilde{a}(x,\xi)\nabla u(x,\xi)\big)+(\nabla u(x,\xi))^T\tilde{b}(x,\xi)&=&\tilde{f}(x,\xi)&\forall x\in D\\
\\
u(x,\xi)&=&0&\forall x\in\partial D.\\
\end{array}
\right.
\end{equation}

%
%
%
%
%
%
%
%
%
%

\subsubsection{Weak solutions at the stochastic level to \eqref{eq:elldiv} and their Galerkin projection}
\label{sssec:weak}

\medskip
\noindent
We now introduce the bilinear form
\begin{equation}
\label{eq:def-A}
\forall u,v\in\cV,\quad\cA(u,v)=\E_\xi\Big( \int_D(\nabla u(\cdot,\xi))^T\tilde{a}(\cdot,\xi)\nabla v(\cdot,\xi)
+ \int_D(\nabla u(\cdot,\xi))^T\tilde{b}(\cdot,\xi)v(\cdot,\xi) \Big)
\end{equation}
and the linear form
$$
\forall v\in\cV,\quad F(v)=\E_\xi\Big(\int_D\tilde{f}(\cdot,\xi)v(\cdot,\xi)\Big)=\langle \tilde{f}, v\rangle_\cH.$$
We introduce now assumptions which ensure  the continuity and the coercivity of the bilinear form $\cA$.

\vspace{0.4cm}
{\bf Assumption 7 (A7).} There is a constant $\tilde{\Lambda}$ s.t.
$$
\forall 1\leq i,j\leq d, \;\;\forall x\in\bar{D},\;\forall z\in\Xi,\quad |\tilde{a}_{ij}(x,z)|\leq\tilde{\Lambda}.$$

\vspace{0.4cm}
{\bf Assumption 8 (A8).} There is a constant $\tilde{\lambda}>0$ s.t. 
$$
\forall y\in\R^d,\;\;\forall x\in\bar{D},\;\forall z\in\Xi,\quad y^T\tilde{a}(x,z)y\geq \tilde{\lambda} |y|^2.$$
 
 \vspace{0.4cm}
 Let us now  recall that thanks to
Poincaré's inequalities the norm 

$v\mapsto \Big(\int_D|\nabla v|^2\Big)^{1/2}$
is equivalent to$||\cdot||_{H^1(D)}$ on $H^1_0(D)$.
More precisely
\begin{equation}
\label{eq:equiv-norm}
\forall v\in H^1_0(D),\quad \Big(\int_D|\nabla v|^2\Big)^{1/2}\leq ||u||_{H^1(D)}\leq C(d,|D|)\Big(\int_D|\nabla v|^2\Big)^{1/2},
\end{equation}
with 
\begin{equation}
\label{eq:const-poinca}
C(d,|D|)=\sqrt{1+\Big( \frac{d\,\Gamma(d/2)}{2\pi^{d/2}}|D|\Big)^{1/d}}
\end{equation}
where $\Gamma(\cdot)$ denotes the Gamma function and $|D|$ is the volume of $D$ \citep[cf][, p164]{trudinger}.
 With this in mind we introduce the following assumption.
 
 \vspace{0.3cm}

{\bf Assumption 9 (A9).} Assume (A8) and
$$
\forall x\in \bar{D}, \;\;\forall z\in\Xi,\quad |\tilde{b}(x,z)|\leq \tilde{M},$$
 with $\tilde{M}<\dfrac{\tilde{\lambda}}{\sqrt{2C(d,|D|)}}$ where $C(d,|D|)$ is the constant defined in \eqref{eq:const-poinca}.
 
 \medskip
 
 \vspace{0.3cm}
 
We have the following result:
 
 \begin{lemma}
 \label{lem:sol-faible-ell}
 Assume (A7)-(A9). There exists a unique  weak solution $u$ at the stochastic level  of \eqref{eq:elldiv}, in the sense that $u\in \cV$ and satisfies
\begin{equation*}
\forall v\in\cV,\quad \cA(u,v)=F(v).
\end{equation*}
This solution $u$ can be approximated by its Galerkin projection $u^{N,P}$, which is the unique element in $\cV^{N,P}$ that satisfies
\begin{equation}
\label{eq:galell}
\forall v\in\cV^{N,P},\quad \cA(u^{N,P},v)=F(v).
\end{equation}
More precisely we have
 $$
||u^{N,P}-u||_\cV\xrightarrow[N\to\infty,P\to\infty]{}0.
$$ 
 \end{lemma}
 
 \begin{proof}
 See the Appendix.
 \end{proof}
 
 \subsubsection{Approximation of the quantity of interest $\gU$}
 \label{sssec:goal}

\vspace{0.6cm}
\noindent
 The metamodel for $\gU$ is provided by the next theorem:

\vspace{0.3cm}

\begin{theorem}
\label{thm:sumup-ell}

 Assume $a(\cdot,\xi)$ belongs to $C^1(D;\R^{d \times d})$ for a.e. value of $\xi$ and that $\tilde{a}$, $\tilde{b}$ and $\tilde{f}$ are defined by \eqref{eq:deftilde}.
Assume  (A2)-(A5). Assume (A9).  Then the quantity of interest $\gU(x,\xi)$ is a weak solution to \eqref{eq:elldiv}, and can be approximated by $\gU^{N,P}$ the solution of \eqref{eq:galell}, in the sense that
$$
||\gU^{N,P}-\gU||_\cV\xrightarrow[N\to\infty,P\to\infty]{}0.
$$
\end{theorem}

\begin{proof}
See the Appendix.
\end{proof}


\subsubsection{Approximation of Sobol' indices}\label{sssec:sob}

We now discuss how to solve \eqref{eq:galell}, for fixed $N,P\in\N^*$ (we omit some superscripts $N,P$ in what follows). A function 
\begin{equation}
\label{eq:Uvect}
\gU^{N,P}=\sum_{j=1}^{N-1}\sum_{q=0}^PU_j^q\psi_q\phi_j\in\cV^{N,P}
\end{equation}
 satisfies \eqref{eq:galell}  if and only if
$$
\forall \,  1\leq i\leq N-1,\quad\forall \,  0\leq p\leq P,\quad\;\; \cA(\gU^{N,P},\psi_p\phi_i)=F(\psi_p\phi_i).$$
Therefore if and only if the block vector 
$$U=(U^0,\ldots,U^P)^T,\textup{ with } U^q=\left(U_1^q , \ldots , U_{N-1}^q\right)\in\R^{1\times(N-1)},\, 0\leq q\leq P$$ solves
\begin{equation}
\label{eq:systU-ell}
\mathbf{A}U=\mathbf{F}
\end{equation}
where
$\mathbf{A}=(\mathbf{A}_{pq})_{0\leq p,q\leq P}$
is the block matrix defined by 
$$\mathbf{A}_{pq}=\big(\cA(\psi_q\phi_j,\psi_p\phi_i)\big)_{1\leq i,j\leq N-1} \textup{ for any }0\leq p,q\leq P$$ and 
$\mathbf{F}=(\mathbf{F}_{0},\ldots,\mathbf{F}_{P})^T$ is the block vector defined by 
$$\mathbf{F}_{p}=\big(F(\psi_p\phi_1),\ldots,F(\psi_p\phi_{N-1})\big)\in \R^{1\times(N-1)} \textup{ for any } 0\leq p\leq P.$$

We thus have to solve \eqref{eq:systU-ell} and to recover $\gU^{N,P}$ by \eqref{eq:Uvect}, as it is usually done in finite elements methods.

\noindent
 Assume then we have got an approximation $\gU^{N,P}(x,\xi)=\sum_{q=0}^P\gU^N_q(x)\psi_q(\xi)$ of  the random variable~$\gU(x,\xi)$, for $x\in D$
 (here we have denoted 
 $\gU^N_q(x)=\sum_{j=1}^{N-1}U^q_j\phi_j(x)$, for any $0\leq q\leq P$). 
 Then, thanks to the orthonormality of the basis $\{\psi_p\}_{p=0}^P$, the Sobol' index $S_I(\gU)$ can be approximated using Parseval identity~by
 \begin{equation}
 \label{eq:pars-U}
 \frac{ \sum_{q\in K_I} [\gU^N_q(x)]^2 }{ \sum_{q=1}^P[\gU^N_q(x)]^2}
\end{equation}
where $K_I=\{p\in\{1,\ldots,P\}\,:\;\psi_p(\xi)=\psi_p(\xi_I)\,\}$.

\subsection{Computation of $S_I(\gV)$  using the parabolic stochastic PDE}
\label{ssec:para}

The structure of this subsection is similar to the structure of Subsection \ref{ssec:ell} about the elliptic case. Note that throughout the subsection we use the  notation already introduced in Subsubsection \ref{sssec:not-sobolev}. In Subsubsection \ref{sssec:div-form-para} we reformulate \eqref{eq:Pb-para} in divergence form. In Subsubsection \ref{sssec:galer-para} we introduce the notion of weak solution at the stochastic level of the parabolic PDE in divergence form, and explain how to get an approximation 
$(\gV^{m,N,P})_m$ of this weak solution. Here we have to perform a discretization in the space variable $x$,  the uncertainty parameter $\xi$, but also the time variable $t$. As again our quantity of interest $\gV$ turns out to be a weak solution of the stochastic PDE of interest, we thus get an approximation of it that allows the computation on approximated Sobol' indices.

\subsubsection{Parabolic PDE in divergence form}
\label{sssec:div-form-para}

As in Subsection \ref{ssec:ell} we first rewrite \eqref{eq:Pb-para} in divergence form. We assume again that for a.e. value of $\xi$ the function $a(\cdot,\xi)$ is of class $C^1(D;\R^{d\times d})$ and define the coefficients $\tilde{a}$, $\tilde{b}$ as in \eqref{eq:deftilde}.
It is then clear that \eqref{eq:Pb-para} is equivalent to

\begin{equation}
\label{eq:Pb-para-div}
\left\{\begin{array}{rll}
\partial_t v(t,x,\xi)&=& \nabla\cdot \big(\tilde{a}(x,\xi)\nabla v(t,x,\xi)\big)-(\nabla v(t,x,\xi))^T\tilde{b}(x,\xi),\,\forall (t,x)\in(0,T]\times D\\
\\
v(t,x,\xi)&=&0,\quad\forall (t,x)\in (0,T]\times\partial D\\
\\
v(0,x,\xi)&=&f(x),\quad\forall x\in\bar{D}.\\
\end{array}
\right.
\end{equation}

\subsubsection{Weak solutions at the stochastic level of \eqref{eq:Pb-para-div} and their approximation; link with the quantity of interest $\gV$}
\label{sssec:galer-para}

In the same spirit as in Subsubsection \ref{sssec:weak} we will present the notion of weak solution of \eqref{eq:Pb-para-div}
at the stochastic level and of its approximation. 
The weak solution will be searched for in the space 
$$
\Big(
L^2(0,T;H^1_0(D))\cap H^1(0,T;H^{-1}(D))
\Big)\otimes L^2(\Xi,\P_\xi)$$
Here the space $L^2(0,T;H^1_0(D))\cap H^1(0,T;H^{-1}(D))$ denotes the space of those functions $v$ such that~$v$ is in $L^2(0,T;L^2(D))$ (i.e. $\int_0^T\int_D v^2<\infty$), the $\partial_{x_i}v$'s, for $1\leq i\leq d$, are in
 $L^2(0,T;L^2(D))$ too, and $\partial_t v$ is in $L^2(0,T;H^{-1}(D))$ (the derivatives are understood in the distribution sense). 
 Note that we have
 $$
 \Big(
L^2(0,T;H^1_0(D))\cap H^1(0,T;H^{-1}(D))
\Big)\otimes L^2(\Xi,\P_\xi)=L^2(0,T;\cV)\cap H^1(0;T,\cV').$$
We start with the two following results (in Lemma \ref{lem:sol-semifaible-para} the form $\cA$ is the one defined in Subsection \ref{ssec:ell}).

\begin{lemma}
\label{lem:sol-semifaible-para}
Assume (A7)-(A8) and that $\tilde{b}$ is bounded. Assume $f\in L^2(D)$. There exists a unique weak solution $v$ at the stochastic level to \eqref{eq:Pb-para-div}, in the sense that 
$v$ belongs to 
$L^2(0,T;\cV)\cap H^1(0;T,\cV')\cap C(0,T;\cH)$
 and satisfies
\begin{equation}
\label{eq:sol-semifaible-para}
\forall w\in \cV,\, \langle \partial_t v(t,\cdot,\cdot)\,,\,w\rangle_\cH+\cA\big(v(t,\cdot,\cdot)\,,\,w\big)=0\text{ for a.e. }t\in[0,T]\text{ and }
v(0,\cdot)=f.
\end{equation}
\end{lemma}

\begin{proof}
See the Appendix.\end{proof}

\begin{lemma}
\label{lem:sol-classique-semifaible-para}
Assume $a(\cdot,\xi)$ belongs to $C^1(D;\R^{d \times d})$ for a.e. value of $\xi$ and $\tilde{a}$ and $\tilde{b}$  are defined by \eqref{eq:deftilde}.
Assume (A2)-(A6). Then the quantity of interest $\gV(t,x,\xi)$ is a weak solution to~\eqref{eq:Pb-para-div}.
\end{lemma}

\begin{proof}
See the Appendix.
\end{proof}

\begin{remark}
Note that in the above results we have dropped Assumption (A9) in itself and we only have assumed that $\tilde{b}$ is bounded. This because in the parabolic case we do not need the form $\cA$ to be coercive in itself; it suffices to get the coercivity of 
$(u,v)\mapsto\cA(u,v)+\mu\langle u,v\rangle_\cH$ for $\mu$ large enough (see the proof of Lemma \ref{lem:sol-semifaible-para}).
\end{remark}

\paragraph{Galerkin discretisation}
We now aim at approximating the weak solution of \eqref{eq:Pb-para-div}. 
We first consider a Galerkin type discretisation of \eqref{eq:sol-semifaible-para} with respect to the space variable and the uncertain parameter, similar to the one we have used for the elliptic case (the time-discretisation will be treated in a second time). 
\vspace{0.2cm}
We consider the problem of finding $v^{N,P}$ in 
$C^1(0,T;\cV^{N,P})$ satisfying

\begin{equation}
\label{eq:semidiscrete}
\left\{\begin{array}{l}
\forall w\in \cV^{N,P},\,\partial_t v^{N,P}(t,\cdot,\cdot)\,,\,w\rangle_\cH+\cA\big(v^{N,P}(t,\cdot,\cdot)\,,\,w\big)=0\text{ for a.e. }t\in[0,T],\\
v^{N,P}(0,\cdot)=f^{N,P}.
\end{array}
\right.
\end{equation}
Here $f^{N,P}$ is an element of $\cV^{N,P}$, and the sequence $(f^{N,P})$ satisfies 
\begin{equation}
\label{eq:fn-to-f}
||f^{N,P}-f||_\cH\xrightarrow[N\to\infty,P\to\infty]{}0.
\end{equation}
Let us now rewrite \eqref{eq:semidiscrete} in matrix form. We have
 $$v^{N,P}(t,x,\xi)=
\sum_{j=1}^{N-1}\sum_{q=0}^PV^q_j(t)\psi_q(\xi)\phi_j(x)$$
and denote $V(t)$ the block vector 
$$(V^0(t),\ldots,V^P(t))^T, \textup{with }V^q(t)=\left(V_1^q(t) , \ldots , V_{N-1}^q(t)\right)\in\R^{1\times(N-1)}, \,0\leq q\leq P.$$
We denote $\mathbf{f}=(\mathbf{f}_0,\ldots,\mathbf{f}_P)^T$  the block vector defined
by 
$$\mathbf{f}_p=(f^{N,P}_{1,p},\ldots,f^{N,P}_{N-1,p}), \,0\leq p\leq P, \textup{ with }f^{N,P}=\sum_{i=1}^{N-1}\sum_{p=0}^Pf^{N,P}_{i,p}\psi_p\phi_i.$$ 
Using then the block matrix 
$\mathbf{A}$ defined in Subsubsection \ref{sssec:sob} and introducing $\mathbf{M}=(\mathbf{M}_{pq})_{0\leq p,q\leq P}$ the block matrix defined by 
$\mathbf{M}_{pq}=\big(\langle\psi_q\phi_j,\psi_p\phi_i\rangle_\cH\big)_{1\leq i,j\leq N-1}$ for any $0\leq p,q\leq P$, it is clear that
\eqref{eq:semidiscrete} is equivalent to
\begin{equation}
\label{eq:syst-diff}
\mathbf{M}V'(t)+\mathbf{A}V(t)=0, \quad V(0)=\mathbf{f}.
\end{equation}
Equation \eqref{eq:syst-diff} is a system of Ordinary Differential Equations (ODE). Note that we know by the Cauchy-Lipschitz theorem  that this system has a solution in $C^1(0,T;\R^{(N-1)\times(P+1)})$ \citep[see Theorem VII.3 in ][]{brezis};  we use here the fact that 
$\mathbf{M}$ is invertible and constant). It ensures the existence of $v^{N,P}$ in~$C^1(0,T;\cV^{N,P})$ that solves \eqref{eq:semidiscrete}.

Under for example (A7)-(A8) and $\tilde{b}\equiv 0$ it can be proved \citep[see Corollary 7.4-1 in ][]{raviart} that $||v^{N,P}(t,\cdot,\cdot)-v(t,\cdot,\cdot)||_\cH\to 0$, as $N\to\infty,P\to\infty$, for any $t\in[0,T]$, where $v$ is the weak solution to~\eqref{eq:Pb-para-div}. Thus $v^{N,P}$ is an approximation of $v$. However we will not access to $v^{N,P}$ itself but to its approximation through a time finite difference scheme for \eqref{eq:syst-diff}.

\vspace{0.2cm}

Indeed, in a second time, we introduce a $\theta$-scheme for \eqref{eq:syst-diff} ($\theta\in[0,1]$). We choose $M\in\N^*$ and denote 
$\Delta t=\frac T M$ and $t_m=m\Delta t$ for any $0\leq m\leq M$. For any $0\leq m\leq M$, define $V[m] \in \mathbb{R}^{(N-1)\times(P+1)}$ as the solution of:
$$
\frac{1}{\Delta t}\mathbf{M}(V[m+1]-V[m])+\mathbf{A}(\theta V[m+1]+(1-\theta)V[m])=0,\quad V[0]=\mathbf{f},
$$
which can be written as:
\begin{equation}
\label{eq:theta-scheme}
(\mathbf{M}+\theta\Delta t\mathbf{A})V[m+1]=(\mathbf{M}-(1-\theta)\Delta t\mathbf{A})V[m],
\quad V[0]=\mathbf{f}.
\end{equation}
We then have the following result:

\begin{lemma}
\label{lem:conv-schema-para}
Assume (A7)-(A8) and that $\tilde{b}$ is bounded. Assume $f\in H^1_0(D)$
and that we have a sequence of functions
$f^{N,P}\in\cV^{N,P}$  satisfying $||f^{N,P}-f||_\cH\to 0$ as $N,P\to\infty$.

Consider then $v$ the weak solution of~\eqref{eq:Pb-para-div}
and $v^{m,N,P}=\sum_{j=1}^{N-1}\sum_{q=0}^PV^q_j[m]\psi_q\phi_j$, where the $V^q_j[m]$'s are computed by
\eqref{eq:theta-scheme}.

Assume $\theta=1/2$ and $u\in C^1(0,T;\cV)\cap C^3(0,T;\cH)$. Then
$$
\sup_{0\leq m\leq M}||v^{m,N,P}-v(t_m,\cdot,\cdot)||_\cH\xrightarrow[M\to\infty,N\to\infty,P\to\infty]{}0.
$$
\end{lemma}

\begin{proof}
See the Appendix.
\end{proof}

\begin{remark}
\label{rem:crank}
In fact the scheme \eqref{eq:theta-scheme} is unconditionally stable for $\theta\in[\frac 1 2,1]$.  If
$\theta=1/2$ and $u\in C^1(0,T;\cV)\cap C^3(0,T;\cH)$ it is consistent and thus convergent (Lax principle), with order  $(\Delta t)^2$ in time. See the proof 
of Lemma~\ref{lem:conv-schema-para}
in the Appendix for details. For $\theta=1/2$, this scheme is known as the Crank-Nicholson scheme.
\end{remark}


Gathering the above results we are led to the following theorem, which provides the metamodel for~$\gV$.

\begin{theorem}
\label{thm:sumup-para}
Assume $a(\cdot,\xi)$ belongs to $C^1(D;\R^{d \times d})$ for a.e. value of $\xi$ and that $\tilde{a}$ and $\tilde{b}$  are defined by \eqref{eq:deftilde}.
Assume (A2)-(A5) and that $f$ satisfies (A6). 

Let the sequence $f^{N,P}$ be defined as in Lemma \ref{lem:conv-schema-para}. Choose $M\in\N^*$ and compute 
$\Big(\gV^{m,N,P}=\sum_{j=1}^{N-1}\sum_{q=0}^PV^q_j[m]\psi_q\phi_j\Big)_{0\leq m\leq M}$ by solving
\eqref{eq:theta-scheme} with $\theta=1/2$.

If the quantity of interest $\gV(t,x,\xi)$ is in $C^1(0,T;\cV)\cap C^3(0,T;\cH)$ it can be approximated by $(\gV^{m,N,P}(x,\xi))_{0\leq m\leq M}$ in the sense that
$$
\sup_{0\leq m\leq M}||\gV^{m,N,P}-\gV(t_m,\cdot,\cdot)||_\cH\xrightarrow[M\to\infty,N\to\infty,P\to\infty]{}0.
$$
\end{theorem}

\begin{remark}\label{interp}
For any $N \in \mathbb{N}^*$, $P \in \mathbb{N}$, $f_{N,P}$ can be, e.g., built as an interpolation function or as a projection on $\cV^{N,P}$.
\end{remark}

\begin{remark}\label{reapprox}
Note that the assumption that $\gV(t,x,\xi)$ is in $C^1(0,T;\cV)\cap C^3(0,T;\cH)$ is sufficient but not necessary. If $f$ is in $\mathcal{C}^{\infty}$, with $f$ and all its derivatives vanishing at the boundary $\partial D$, then the aforementioned assumption is satisfied \citep[see, e.g., Theorem VII.5.][]{brezis}. This last assumption is not restrictive in the sense that similar arguments to the ones in Remark \ref{rem:A6} can be applied in case it is not satisfied.
\end{remark}

\noindent
Then, as in \eqref{eq:pars-U}, the Sobol' index $S_I(\gV(T,x,\xi))$ can be approximated by
 \begin{equation}
 \label{eq:pars-U}
 \frac{ \sum_{q\in K_I} [\gV^{M,N}_q(x)]^2 }{ \sum_{q=1}^P[\gV^{M,N}_q(x)]^2}
\end{equation}
where 
$\gV^{M,N}_q(x)=\sum_{j=1}^{N-1}V^q_j[M]\phi_j(x)$ for $0\leq q\leq P$.

\section{Validation and illustration: numerical experiments on a one-dimensional Ornstein-Uhlenbeck process}
\label{sec:OU}

In this section we choose a simple example, namely a one-dimensional Ornstein-Uhlenbeck (OU) process whose coefficients depend on uncertain parameters, in order to validate and illustrate the PDE based approach we propose. The results we obtain are promising and motivate further work to extend the method to more complex settings.

More precisely our toy model is given by the SDE
\begin{equation}
\label{eq:OU}
dX_t=-\alpha(\xi)X_tdt+\sigma(\xi)dW_t,\quad t\geq 0,\quad X_0=x,
\end{equation} 
where $W$ is a one-dimensional standard Brownian motion and $x\in\R$.

We choose $\xi=(\xi_1,\xi_2)^T$ with $\xi_1,\xi_2$ independent, distributed as $\cU([0,1])$, and independent of $W$, and define
$$
\alpha(\xi)=\mu_1+\sqrt{3}\sigma_1(2\xi_1-1)\quad\text{ and }\quad\sigma(\xi)=\mu_2+\sqrt{3}\sigma_2(2\xi_2-1)$$
with $\sigma_i>0$, $\mu_i>\sqrt{3}\sigma_i$, $i=1,2$. 
Note that the probability distributions of $\alpha(\xi)>0$ and $\beta(\xi)>0$ are then given by
\begin{equation}
\label{eq:lawcoeff}
\alpha(\xi)\sim\cU\big( [\mu_1-\sqrt{3}\sigma_1,\mu_1+\sqrt{3}\sigma_1] \big)
\quad\text{ and }\quad
\sigma(\xi)\sim\cU\big( [\mu_2-\sqrt{3}\sigma_2,\mu_2+\sqrt{3}\sigma_2] \big).
\end{equation}

Note that the truncated orthonormal basis $\{\Psi_q\}_{q=0}^P$ is then chosen as the tensorized Legendre polynomials basis.

To start with, we have the following straightforward result.

\begin{proposition}
Consider the stochastic differential equation described by \eqref{eq:OU}.  
Assume $\xi$ and $\alpha(\xi)$, $\sigma(\xi)$ are defined as above.

Assumptions A1 and A2 are then trivially satisfied. 
In addition the coefficients in \eqref{eq:OU} satisfy A3 and we thus know that \eqref{eq:OU} admits a solution.
\end{proposition}
\vspace{0.2cm}

We set $D=(0,10)$ and start with studying $S_I(\gU)$ in Subsection \ref{ssec:OU-U}. We set $\gV(T,x,\xi)=\P^x(T\leq \tau_D|\xi)$ 
(this corresponds to $f\equiv 1$ in \eqref{eq:def-quant-int2})
and will study
$S_I(\gV)$ in Subsection \ref{ssec:OU-V}.

\vspace{0.3cm}

The OU process is a well marked test case. For example, for any $t\geq 0$, the solution to \eqref{eq:OU} is given by
$
X_t=xe^{-\alpha(\xi)t}+\sigma(\xi)\int_0^te^{-\alpha(\xi)(t-s)}dW_s$
\citep[see][Example 5.6.8]{kara}. 

\subsection{Study of $S_I(\gU)$}
\label{ssec:OU-U}

\begin{proposition}
\label{prop:U-OU}
Consider the stochastic differential equation described by \eqref{eq:OU}. Assume that $D=(0,10)$.  
Assume $\xi$ and $\alpha(\xi)$, $\sigma(\xi)$ are defined as above.

In addition to (A1)-(A3), the Assumptions 4 and 5 are also satisfied. Thus, for any $x\in D$, the quantity $\gU(x,\xi)$ is given by the solution of the stochastic elliptic PDE
\begin{equation}
\label{eq:ellOU}
\left\{\begin{array}{rlll}
\frac 1 2 \sigma^2(\xi)\partial^2_{xx}u(x,\xi)-\alpha(\xi)x\,\partial_{x}u(x,\xi)&=&-1&\forall x\in D\\
\\
u(0,\xi)=u(1,\xi)&=&0.&\\
\end{array}
\right.
\end{equation}
\end{proposition}

\begin{proof}
(A4) is satisfied because of \eqref{eq:lawcoeff}. As we are in dimension one in space ($d=1$), Assumption 5 will be trivially met by any segment $(a,b)\subset\R$. Using Theorem~\ref{thm:feyn-U} we get the last statement of the proposition.
\end{proof}

\medskip

In view of Proposition \ref{prop:U-OU} we wish now to use Theorem \ref{thm:sumup-ell} in order to compute an approximation of the quantity of interest $\gU(x,\xi)$, and then compute the Sobol' indices by \eqref{eq:pars-U}.

 Note that \eqref{eq:ellOU} may trivially be rewritten in the form \eqref{eq:elldiv} 
with $\tilde{a}(x,z)=\sigma^2(z)/2$, $\tilde{b}(x,z)=\alpha(z)\,x$ and $\tilde{f}(x,z)=1$, for $x\in \bar{D}$, $z\in\Xi$. In order to check (A9) we have to impose that
\begin{equation}
\label{eq:coer-OU}
10(\mu_1+\sqrt{3}\sigma_1)<\frac{(\mu_2-\sqrt{3}\sigma_2)^2}{2(24)^{1/4}}.
\end{equation}
Then we have the following result.

\begin{proposition}
\label{prop:solU-OU}
In the context of Proposition \ref{prop:U-OU}, assume that $\sigma_1$, $\sigma_2$, $\mu_1$ and $\mu_2$ satisfy \eqref{eq:coer-OU}.
Then $\gU(x,\xi)$ can be approximated by the solution $\gU^{N,P}$ of \eqref{eq:galell} with, $\forall u,v\in\cV$,
$$\cA(u,v)=\E_\xi\Big( \int_0^{10}\frac{\sigma^2(\xi)}{2}\partial_x u(x,\xi)\partial_x v(x,\xi)dx
+ \int_0^{10}\alpha(\xi)\,x\,\partial_x u(x,\xi))  v(x,\xi)dx \Big)$$
and $F(v)=\E_\xi\big(\int_0^{10}v(x,\xi)\big)dx$, $\forall v\in\cV$.
\end{proposition}

\vspace{1cm}
Let $N,P\in\N^*$ fixed. Denote $x_i=10\frac{i}{N}$, $1\leq i\leq N$, and $h=10/N$. We use as finite element basis functions the $\phi_i$, 
$1\leq i\leq N-1$, defined by $\phi_i(x)=\phi(\frac{x-x_i}{h})$ with $\phi(x)=\1_{|x|\leq 1}(1-|x|)$. Note that $\phi_1(0)=\phi_{N-1}(10)=0$.

In order to follow the program proposed in Proposition \ref{prop:solU-OU} we have then to solve~$\mathbf{A}U=\mathbf{F}$, with $\mathbf{A}$ and $\mathbf{B}$ defined as after Eq. 
\eqref{eq:systU-ell}. Here simple computations show that ($\otimes$ stands for the Kronecker product)
$$
\mathbf{A}=\Sigma\otimes A+ \beta\otimes B$$
with
$$
\Sigma=\Big( \E_\xi\big[\frac{\sigma^2(\xi)}{2}\psi_q(\xi)\psi_p(\xi)\big] \Big)_{0\leq p,q\leq P}\text{ and }
\quad \beta=\Big( \E_\xi\big[\alpha(\xi)\psi_q(\xi)\psi_p(\xi)\big] \Big)_{0\leq p,q\leq P}\, ,$$
$$
A=\Big( \int_0^{10} \partial_x \phi_j  \partial_x \phi_i  \Big)_{1\leq i,j\leq N-1}
=\frac 1 h \left(
\begin{array}{ccccc}
2&-1&&&0\\
-1&2&&&\\
&&\ddots&&\\
&&&2&-1\\
0&&&-1&2\\
\end{array}
\right)\, ,
$$
$$B=\Big( \int_0^{10} x\partial_x \phi_j(x)  \phi_i(x)\,dx  \Big)_{1\leq i,j\leq N-1}
=  \left(
\begin{array}{ccccc}
-\frac h 3&\frac{x_1}{2}+\frac h 6&&&0\\
-\frac{x_2}{2}+\frac h 6&-\frac h 3&&&\\
&&\ddots&&\\
&&&-\frac h 3&\frac{x_{N-2}}{2}+\frac h 6\\
0&&&-\frac{x_{N-1}}{2}+\frac h 6&-\frac h 3\\
\end{array}
\right).
$$
Note that $\Sigma$, $\beta$ and $A$ are symmetric but not $B$. Besides $\mathbf{F}=(\mathbf{F}_p)_{0\leq p\leq P}$ is simply given by
$\mathbf{F}_0=h(1,\ldots,1)^T\in\R^{N-1}$ and $\mathbf{F}_p=0$ for $1\leq p\leq P$.

\vspace{0.4cm}

We will compare our method with crude Monte Carlo estimator. In our Monte Carlo estimation of the exit time $\tau_D$ we use a Euler scheme and a boundary shifting method to be described in \cite{gobet-livre-num}.

\vspace{0.1cm}

More precisely the function $\widehat{\cF}_\gU$  we use (see Subsubsection \ref{sssec:MC}) can be described as:
\vspace{0.4cm}

\hspace{-1.5em}
\fbox{\parbox{\textwidth}{
{\bf INPUT:} the brownian path $W$, the uncertain parameter $\xi$ and the time step $\delta t$.

\vspace{0.3cm}

{\bf COMPUTE:} $\sigma(\xi)=\mu_2+\sqrt{3}\sigma_2(2\xi_2-1)$, $\alpha(\xi)=\mu_1+\sqrt{3}\sigma_1(2\xi_1-1)$
and $c(\xi)=0.5826\,\sigma(\xi)\,\sqrt{\delta t}$.

\vspace{0.3cm}
{\bf INITIALIZE:} $X=x$, $t=0$ and $d=100$ (in fact any value that is greater than $c(\xi)$, for any $\xi$).

\vspace{0.3cm}
\hspace{1cm}{\bf While} $\big\{ \;\;X\in(0,10)\quad \text{and}\quad c(\xi)<d  \big\}$

\hspace{1.5cm}{\bf Do:} Set $X=X+\sqrt{\sigma(\xi)}\sqrt{\delta t}(W_{t+\delta t}-W_t)-\alpha(\xi)\,X\,\delta t$

\hspace{2.8cm} $t=t+\delta t$

\hspace{2.8cm} $d=\mathrm{dist}(X,\partial D)$

\hspace{1cm}{\bf Endwhile}

\vspace{0.4cm}
{\bf OUTPUT:} the time $t$ that is reached when exiting the while loop.}}

%

\medskip

In Table \ref{tab:1} we illustrate the efficiency of our approach with $\mu_1=1$, $\mu_2=9$ and $\sigma_1=\sigma_2=0.2$, and with the starting point $x=5$.

\begin{table}[h!]
 \begin{center}
\begin{tabular}{cccc}
\hline
 & Galerkin  &  Double MC   \\
 \hline
$S_1(\gU)$ &  ($N=1000$, $P=10$)   0.0253    &   ($N=M=10^4$, $\delta t = 10^{-3}$) 0.024927  \\
\hline
$S_2(\gU)$ &  ($N=1000$, $P=10$)   0.9747    &  ($N=M=2\times 10^4$, $\delta t = 10^{-4}$)  0.971277   \\
\end{tabular}
\caption{Sobol' indices for $\gU(x,\xi)$, $x=5$, in the Ornstein-Uhlenbeck example, with $\mu_1=1$, $\mu_2=9$ and $\sigma_1=\sigma_2=0.2$.}
\label{tab:1} 
\end{center}
\end{table}


\subsection{Study of $S_I(\gV)$}
\label{ssec:OU-V}

Recall that we consider $\gV(T,x,\xi)=\P^x(T\leq \tau_D|\xi)$ 
(this corresponds to $f\equiv 1$ in \eqref{eq:def-quant-int2}).
Note that $f\equiv 1$ do not satisfy all the requirements of (A6). In particular it does not satisfy the uniform Dirichlet boundary condition. However it is possible to approach $f$ by $f_\varepsilon$ satisfying (A6). Using the construction based on convolution arguments suggested in Remark \ref{rem:A6}, we get $f_{\varepsilon} \equiv f$ on $K_{\varepsilon}$ with $K_{\varepsilon}$ a compact subset of $D$ such that $|D \setminus K_\varepsilon| \leq \varepsilon$.
A consequence is that, numerically, if we choose a discretization step small enough, $f_\varepsilon$ will coincide with $f$.

\vspace{0.2cm}
We take the same finite element basis $\phi_i$, $1\leq i \leq N-1$ as in Subsection \ref{ssec:OU-U}. We will use $M$ time steps. We take $\epsilon < h$ so that 
$f_\epsilon$ can be simply interpolated by
$$
f^{N,P}=\sum_{i=1}^{N-1}f_\epsilon(x_i)\phi_i=\sum_{i=1}^{N-1}\phi_i$$
and thus the block vector $\mathbf{f}=(\mathbf{f}_p)_{0\leq p\leq P}$ is given 
by
$\mathbf{f}_0=(1,\ldots,1)^T\in\R^{N-1}$ and $\mathbf{f}_p=0$ for $1\leq p\leq P$.
The block matrix $\mathbf{A}$ remains as in Subsection \ref{ssec:OU-U}. In order to solve \eqref{eq:theta-scheme} (with
$\theta=\frac 1 2$) it remains to compute $\mathbf{M}$. But simple computations show that
$$
\mathbf{M}=\mathcal{M} \otimes M_1$$
with 
$\mathcal{M} = \Big( \E_\xi\big[\psi_q(\xi)\psi_p(\xi)\big] \Big)_{0\leq p,q\leq P}$
and 
$$
M_1=\Big( \int_0^{10}  \phi_j(x)  \phi_i(x)\,dx  \Big)_{1\leq i,j\leq N-1}
=  \left(
\begin{array}{ccccc}
2h/3 & h /6&&&0\\
 h/ 6&2h / 3&&&\\
&&\ddots&&\\
&&&2h / 3& h / 6\\
0&&& h /6& 2h / 3 \\
\end{array}
\right).
$$
We compare our method with crude Monte Carlo estimator. 
That is to say we use in \eqref{eq:MC} a function~$\widehat{\cF}_\gV$ that is constructed in the same spirit as $\widehat{\cF}_\gU$ in Subsection
\ref{ssec:OU-U} (in particular we use again a boundary shifting method).
In Table \ref{tab:2} we illustrate the efficiency of our approach with $\mu_1=1$, $\mu_2=2$ and $\sigma_1=\sigma_2=0.2$, and with the starting point $x=1$ and the time horizon $T=0.3$.

\begin{table}[h!]
 \begin{center}
\begin{tabular}{cccc}
\hline
 & Galerkin / Crank-Nicholson &  Double MC    \\
 \hline
$S_1(\gV)$ &  ($M=300$, $N=1000$, $P=10$)   0.0961   &   ($N=M=5\times 10^4$, $\delta t = 6\times 10^{-4}$) 0.095812 \\
\hline
$S_2(\gV)$ &  ($M=300$, $N=1000$, $P=10$)   0.9039    &  ($N=M=5\times 10^4$, $\delta t = 6\times 10^{-4}$)  0.903182  \\
\end{tabular}
\caption{Sobol' indices for $\gV(T,x,\xi)$, $x=1$, $T=0.3$ in the Ornstein-Uhlenbeck example, with $\mu_1=1$, $\mu_2=2$ and $\sigma_1=\sigma_2=0.2$.}
\label{tab:2} 
\end{center}
\end{table}



\section{Appendix}
\label{sec:app}

\subsection{Integrability of the stochastic quantities of interest}

\begin{proof}[Proof of Lemma \ref{lem:X-L2}]
Let $q\in\N^*$ and $t\geq 0$. Using Problem 5.2.15 in \cite{kara}, we can claim that, under (A3), for any $z\in\Xi$
\begin{equation}
\label{eq:m2-maj}
\int_C\,|X_t(\theta,z)|^q\W(d\theta)\leq C(1+|x|^q)\exp\big(Ct\big),
\end{equation}
where  $C=C(t,M)$.  As the constant $M$ in  \eqref{eq:growth} is uniform in $z$, the constant $C$ does not depend on $z$ either. Integrating \eqref{eq:m2-maj} over $\Xi$ and against   $\P_\xi$ we get the desired result.
\end{proof}

\begin{proof}[Proof of Lemma \ref{lem:tau-L1}]
Here we revisit the proof of Lemma 5.7.4 in \cite{kara} and give details for the sake of completeness. Thanks to (A4), we have for example
$a_{11}(x,z)\geq \lambda$ for any $x=(x_1, \ldots , x_d)\in\R^d$, $z\in\Xi$. We note $q=\min_{x\in \bar{D}}x_1$ and $b$ a constant s.t.
$|b(x,z)|\leq b<\infty,\,\forall x\in D,\,\forall z\in\Xi$.

Such a finite constant $b$ exists thanks to \eqref{eq:growth}. Set now $\nu>2b/\lambda$ and consider the function
$h(x)=-\mu\exp(\nu x^1)$, $x\in D$, where the constant $\mu$ will be determined later. This function is of class $C^\infty(D)$ and satisfies for any $x\in D$ and $z\in\Xi$,
$$
\begin{array}{ll}
-\frac1 2 \sum_{i,j=1}^da_{ij}(x,z)\partial^2_{x_ix_j}h(x)-\sum_{j=1}^db_j(x,z)\partial_{x_i}h(x)&\\\\
\quad \quad\quad= \mu e^{\nu x^1}(\frac 1 2\nu^2a_{11}(x,z)+\nu b_1(x,z))
\geq\frac 1 2 \mu\nu \lambda e^{\nu q}\big( \nu-\frac{2b}{\lambda} \big).&
\end{array}
$$
Choosing $\mu$ sufficiently large, we have for any $x\in D$, any $z\in\Xi$:
$$\frac1 2 \sum_{i,j=1}^da_{ij}(x,z)\partial^2_{x_ix_j}h(x)+\sum_{j=1}^db_j(x,z)\partial_{x_i}h(x)\leq -1.$$
Let us write $\tau_D(\omega)=\tau_D(\theta,z)$. As $X_0=x \in D$, using It\^o Formula and taking the expectation against $\W$ we get:
$$
\forall z\in\Xi\,\int_C(t\wedge\tau_D(\theta,z)\W(d\theta)\leq h(x)-\int_Ch(X_{t\wedge\tau_D}(\theta,z))\W(d\theta)\leq 2\max_{y\in\bar{D}}|h(y)| \, .$$
Note that in the above computation, the expectation of the stochastic integral vanishes, as $h$ is bounded with bounded derivatives. Then, integrating over $\Xi$ and against $\P_\xi$ the above inequality we get the result by monotone convergence.
\end{proof}

\begin{proof}[Proof of Lemma \ref{coro:gU-L2}]
According to Theorem \ref{thm:feyn-U}, the function $\gU(\cdot,\xi)$ solves \eqref{eq:Pb-ell} (for a.e. fixed value of $\xi$). We now use a maximum principle argument to get an a priori bound on $\gU(\cdot,\xi)$. 
Without loss of generality we assume that $D$ lies in the slab $0<x^1<\delta$ for a certain $0<\delta<\infty$ (the general case can be recovered by  translation arguments). 
According to Theorem 3.7 in \cite{trudinger} and its proof we have
\begin{equation}
\label{eq:maximum}
\sup_{x\in\bar{D}}|\gU(x,\xi)|\leq \frac{C'(\xi)}{\lambda}
\end{equation}
where $C'(\xi)=e^{\alpha(\xi)\delta}-1$  with $\alpha(\xi)$ a quantity chosen s.t. $\alpha(\xi)\geq 1+\frac 1 \lambda \sup_{x\in D}|b(x,\xi)|$ (note that the uniform ellipticity constant $\lambda$ in (A4) does not depend on $z$). Then, thanks to the uniformity in $z$ of the constant $M$ in
\eqref{eq:growth} and to the boundedness of $D$, the quantity $\alpha(\xi)$ may be chosen independent on $\xi$. Therefore $C'(\xi)=C'$  does not depend on $\xi$, and thus we get the result with $C=C'/\lambda$.
\end{proof}

\subsection{Results on stochastic partial differential equations}
\begin{proof}[Proof of Lemma \ref{lem:sol-faible-ell}]
As $\tilde{f}$ is uniformly bounded in $x$ and $z$, it is easy to prove that the form $F$ is continuous on $\cV$ (with the help of Cauchy-Schwarz's inequality).
Thanks to (A7) and (A9), the bilinear form $\cA$ is continuous on $\cV$. Thanks to (A8)-(A9) one may check the coercivity of 
$\cA$. Indeed, using Lemma~8.4 in \cite{trudinger}, one gets for any $z\in\Xi$,
\begin{eqnarray}
\label{eq:coer-ell}
\int_D(\nabla v(\cdot,z))^T\tilde{a}(\cdot,z)\nabla v(\cdot,z)
+ \int_D(\nabla v(\cdot,z))^T\tilde{b}(\cdot,z)v(\cdot,z)&\nonumber \\
\quad \quad \quad  \geq\frac{\tilde{\lambda}}{2}\int_D|\nabla v(\cdot,z)|^2-\tilde{\lambda}\nu^2\int_Dv^2(\cdot,z)&
\end{eqnarray}
where $\nu=\dfrac{\tilde{M}}{\tilde{\lambda}}$. Using now the definition of $||\cdot||_{H^1(D)}$, Inequality \eqref{eq:equiv-norm}, and integrating \eqref{eq:coer-ell} over $\Xi$ against $\P_\xi$ we get
$$
\cA(v,v)\geq \Big( \frac{\tilde{\lambda}}{2C(d,|D|)}-\tilde{\lambda}\nu^2\Big)||v||^2_\cV$$
(we recall that $C(d,|D|)$ is defined by \eqref{eq:const-poinca}).
As $\tilde{M}<\dfrac{\tilde{\lambda}}{\sqrt{2C(d,|D|)}}$, we have that 
$\dfrac{\tilde{\lambda}}{2C(d,|D|)}-\tilde{\lambda}\nu^2>0$.
The existence of a unique weak solution at the stochastic level then follows from the Lax-Milgram theorem.

Now, as $F$ and $\cA$ are continuous, as $\cA$ is coercive and as $\cV^{N,P}\subset\cV$, the existence of a unique solution to \eqref{eq:galell} follows from Lax-Milgram theorem again, but applied this time on~$\cV^{N,P}$. In order to prove the convergence result, we first notice that thanks to the Céa lemma we have for any $N,P\in\N^*$
$$
||u-u^{N,P}||_\cV\leq \tilde{C}\min_{v\in\cV^{N,P}}||u-v||_\cV,
$$
where the constant $\tilde{C}$ depends on $\tilde{\lambda}$, $\tilde{\Lambda}$, $\tilde{M}$ and $C(d,|D|)$. Second, we recall that
$\cV^{N,P}\subset\cV^{N+1,P+1}$  for any $N,P$, and that $\bigcup_{N,P}\cV^{N,P}$ is dense in $\cV$. This is sufficient in order to prove that $\min_{v\in\cV^{N,P}}||u-v||_\cV\to 0$, as $N,P\to\infty$, and the result follows.
\end{proof}

\begin{proof}[Proof of Theorem \ref{thm:sumup-ell}]
 By Corollary \ref{cor:sol-classique}, $\gU(x,\xi)$ is a classical solution of \eqref{eq:Pb-ell}.  Then, using Theorem 6.6 in \cite{trudinger} we get that, for a.e. value of $\xi$
$$
\forall 1\leq i\leq d,\quad \sup_{x\in\bar{D}}|\partial_{x_i}\gU(x,\xi)|\leq C''(\sup_{x\in\bar{D}}|\gU(x,\xi)|+1),$$
with $C''$ depending on $d$, $M$, $|D|$ and $\lambda$, but not of $\xi$. But, as already noticed in the proof of Lemma \ref{coro:gU-L2}, we have
$\sup_{x\in\bar{D}}|\gU(x,\xi)|\leq C$ with $C$ depending on $M$, $|D|$ and $\lambda$. Thus $\gU(x,\xi)$ belongs to the space~$\cV$.

As noticed in Subsection \ref{ssec:ell} any classical solution of \eqref{eq:Pb-ell}
is a classical solution of \eqref{eq:elldiv}. Let $v\in C^1_c(D)\otimes L^2(\Xi,\P_\xi)$. Multiplying the first line of \eqref{eq:elldiv}
by $v$, integrating over $D\times \Xi$ against $dx\otimes \P_\xi(dz)$, and performing integration by parts w.r.t  the variable $x$, we get that $\cA(\gU,v)=F(v)$. Thanks to Equation~\eqref{eq:growth} in~(A3),  the coefficients $\tilde{a}$ and $\tilde{b}$ are obviously bounded uniformly in $x$ and $z$, because $D$ is bounded, and we have therefore (A7) and the continuity of the form $\cA$. Using then density arguments we get $\cA(\gU,v)=F(v)$ for all $v\in\cV$. (A8) is a consequence of (A4), and we are under (A9). Thus we may get the approximation result by applying Lemma \ref{lem:sol-faible-ell}.
\end{proof}

\begin{proof}[Proof of Lemma \ref{lem:sol-semifaible-para}]
In the proof of Lemma~\ref{lem:sol-faible-ell}  we have seen that (A7) and the boundedness of $\tilde{b}$ imply that the form $\cA$ is continuous 
 on $\cV$. 

We now prove that we can find $\mu\geq 0$ large enough such that the form $\cA_\mu:(u,v)\mapsto\cA(u,v)+\mu\langle u,v\rangle_\cH$ is coercive on $\cV$. Let us denote $\tilde{M}$ the upper bound for $\tilde{b}$ (i.e. $|\tilde{b}(x,z)|\leq \tilde{M}$ for any $x\in\bar{D}$, $z\in\Xi$). Proceeding as for Equation \eqref{eq:coer-ell} (that we integrate over $\Xi$ against $\P_\xi$) we get that for any $v\in\cV$
$$
\cA_\mu(v,v)\geq \frac{\tilde{\lambda}}{2}\E_\xi\big(\int_D|\nabla v(\cdot,\xi)|^2\big)+(\mu-\frac{\tilde{M}^2}{\tilde{\lambda}})
\E_\xi\big(\int_Dv^2(\cdot,\xi)\big).$$
Note that if $\tilde{M}<\dfrac{\tilde{\lambda}}{\sqrt{2C(d,|D|)}}$ it suffices to take $\mu=0$ (see the proof of Lemma \ref{lem:sol-faible-ell}). If this is not the case we 
take $\mu>\frac{\tilde{M}^2}{\tilde{\lambda}}$ and set
$c=\min\big( \frac{\tilde{\lambda}}{2},\mu-\frac{\tilde{M}^2}{\tilde{\lambda}}\big)$. We then have 
$\cA_\mu(v,v)\geq c||v||^2_\cV$ for any $v\in\cV$.

 This is sufficient to prove the result (see Lemma 7.1.2 and Theorem 7.1.4 in \cite{ern}; see also the earlier reference
\cite{lions-magenes}). 
\end{proof}

\begin{proof}[Proof of Lemma \ref{lem:sol-classique-semifaible-para}]
 By Corollary \ref{cor:sol-classique}, $\gV(t,x,\xi)$ is a classical solution of \eqref{eq:Pb-para}. Using Theorem
5.14 in \cite{lieberman} we get that for a.e. value of $\xi$
$$
\sup_{(t,x)\in [0,T]\times\bar{D}}| \gV(t,x,\xi)|+
\sup_{(t,x)\in [0,T]\times\bar{D}}|\partial_t \gV(t,x,\xi)|
+\sum_{i=1}^d\sup_{(t,x)\in [0,T]\times\bar{D}}|\partial_{x_i} \gV(t,x,\xi)|
$$
is bounded above by some constant $C$ that depends on $M$, $|D|$, $\lambda$ and $d$ but not on $\xi$. We use that the coefficients $a$ and $b$ are Lipschitz continuous and that the initial condition $f$ is in the Hölder space~$H_{2+\alpha}$ for $\alpha=1$; see pp
46-47 of \cite{lieberman} for more details on H\"older spaces and norms in the parabolic setting. The point is that here we can write these properties of $a$, $b$ and $f$ with quantities that are uniform w.r.t. the uncertain parameter $\xi$.
Thus the function $\gV(t,x,\xi)$ belongs to the space $L^2(0,T;\cV)\cap H^1(0;T,\cV')$.
As it is continuous in time and bounded (in particular uniformly w.r.t. $\xi$) it is also in
$C([0,T];\cH)$.

As noticed in Subsubsection \ref{sssec:div-form-para} the function $\gV(t,x,\xi)$ is a classical solution of \eqref{eq:Pb-para-div}.
Multiplying the first line of \eqref{eq:Pb-para-div} by a test function in $C^1_c(D)\otimes L^2(\Xi,\P_\xi)$, integrating over 
$D\times\Xi$  against $dx\otimes \P_\xi(dz)$,
 and using density arguments
we get
$$
\forall w\in\cV,\, \langle \partial_t\gV(t,\cdot,\cdot),w\rangle_\cH+\cA(\gV(t,\cdot,\cdot),w)=0,\,\forall t\in[0,T] \text{ and }
\gV(0,\cdot,\cdot)=f.$$
The result is proved.
\end{proof}

\begin{proof}[Proof of Lemma \ref{lem:conv-schema-para}]
Remember that in the proof of Lemma \ref{lem:sol-semifaible-para} we have seen that we can find $\mu\geq 0$ large enough such that the form $\cA_\mu$ is coercive. Here for simplicity we assume that $\mu=0$, i.e. $\cA$ is coercive (with constant $c$). We claim that this is without loss of generality. Indeed if 
$\mu>0$ we consider~$v_\mu$ the solution of
\begin{equation}
\label{eq:sol-semifaible-para-mu}
\forall w\in \cV,\, \langle \partial_t v_\mu(t,\cdot)\,,\,w\rangle_\cH+\cA_\mu\big(v_\mu(t,\cdot)\,,\,w\big)=0\text{ for a.e. } t\in[0,T]\text{ and }v_\mu(0,\cdot)=f
\end{equation}
(the solution exists because the form $\cA_\mu$ is continuous and coercive). Then it is obvious that the function
$v(t,\cdot)=e^{\mu t}v_\mu(t,\cdot)$ solves \eqref{eq:sol-semifaible-para} (in fact \eqref{eq:sol-semifaible-para} and \eqref{eq:sol-semifaible-para-mu} are equivalent via the change of variable). Thus one may approach $v_\mu$ by the $\theta$-scheme and get
an approximation of $v$ applying again the change of variable.

\vspace{0.2cm}
Then for any $N,P$ the form $\cA$ defines a scalar product on $\cV^{N,P}$, so that for any $u\in\cV$ there is by the Riesz theorem an element $\Pi^{N,P}u$
of $\cV^{N,P}$ such that
$$
\forall w\in\cV^{N,P},\quad \cA(\Pi^{N,P}u,w)=\cA( u,w).$$
The application $\Pi^{N,P}$ is linear and continuous from $\cV$ to $\cV^{N,P}$ and is called the elliptic projection operator. It satisfies
\begin{equation*}
\forall u\in\cV,\quad\forall w\in\cV^{N,P},\quad \cA(u-\Pi^{N,P}u,w)=0.
\end{equation*}
Then, following the proof of Céa's Lemma (Theorem 3.1-2 in \cite{raviart}) one can prove that
$$
||u-\Pi^{N,P}u||_\cV\leq \tilde{C}\min_{w\in\cV^{N,P}}||u-w||_\cV,
$$
where again $\tilde{C}$ depends on the continuity and coercivity constants of $\cA$. Using the monotonicity and density assumptions on the $\cV^{N,P}$'s we thus get that
\begin{equation}
\label{eq:conv-proj}
\forall u\in\cV,\quad ||u-\Pi^{N,P}u||_\cV\xrightarrow[N\to\infty,P\to\infty]{}0.
\end{equation}
In the sequel we note $\Pi$ for $\Pi^{N,P}$ in order to lighten notations.  For any $y\in C(0,T;\cH)$, $t\in[0,T]$ we denote~$y(t):=y(t,\cdot,\cdot)$.

\vspace{0.2cm}
Then for any $0\leq m\leq M$ we set 
$$
e^m_{N,P}=v^{m,N,P}-\Pi v(t_m)$$
(here we follow for example \cite{raviart} §7.5).
 Note that \eqref{eq:theta-scheme} is equivalent to: $\forall \, 0\leq m\leq M-1$, $\forall \, w\in\cV^{N,P}$,
$$
\frac{1}{\Delta t} \langle v^{m+1,N,P}-v^{m,N,P},w\rangle_\cH+\cA\big(\theta v^{m+1,N,P}+(1-\theta)v^{m,N,P}\,,\,w\big)=0,$$
so that by some algebraic computations we see that: $\forall \, 0\leq m\leq M-1$, $\forall w\in\cV^{N,P}$,
\begin{equation}
\label{eq:erreur-sol-faible}
\frac{1}{\Delta t} \langle e^{m+1}_{N,P}-e^{m}_{N,P},w\rangle_\cH+\cA\big(\theta e^{m+1}_{N,P}+(1-\theta)e^{m}_{N,P}\,,\,w\big)
=\langle \varepsilon^m_{N,P}\,,\,w\rangle_\cH
\end{equation}
where $\varepsilon^m_{N,P}\in\cV'$ is defined for any $0\leq m\leq M-1$ by: $\forall \, w\in\cV$,
\begin{equation}
\label{eq:def-eps-m}
\langle \varepsilon^m_{N,P}\,,\,w\rangle_\cH=
-\frac{1}{\Delta t} \langle \Pi v(t_{m+1})-\Pi v(t_m),w\rangle_\cH
-\cA\big(\theta v(t_{m+1})+ (1-\theta)v(t_m)\,,\,w\big).
\end{equation}
In the sequel we denote $e^m$ (resp. $\varepsilon$) for $e^m_{N,P}$ (resp. $\varepsilon^m_{N,P}$).

\vspace{0.2cm}
We now aim at showing that for any $\theta\in[\frac 1 2,1]$ we have the stability result
\begin{equation}
\label{eq:stabilite}
\forall 1\leq m\leq M,\quad ||e^m||_\cH^2\leq ||e^0||_\cH^2+\frac{\Delta t}{c}\sum_{k=1}^{m-1}||\varepsilon^k||_{\cV'}^2.
\end{equation}
Then we will aim at controlling the $||\varepsilon^k||_{\cV'}^2$'s (consistency result; we will only treat the case $\theta=\frac 1 2$). This will allow to get convergence.

\vspace{0.2cm}
\noindent
{\bf Stability.} Here we adapt the energy estimate method to be found pp66-67 in \cite{dautray-vol6}. Let $\theta\in[\frac 1 2,1]$ and fix $0\leq k\leq M-1$.
Taking $w=\theta e^{k+1}+(1-\theta)e^k=:\bar{e}^k$ in \eqref{eq:erreur-sol-faible} we get
\begin{equation}
\label{eq:stab1}
\frac{1}{\Delta t}\langle e^{k+1}-e^{k},\bar{e}^k\rangle_\cH+\cA(\bar{e}_k,\bar{e}_k)=\langle \varepsilon^k\,,\,w\rangle_\cH.
\end{equation}
 Using then the algebraic equality
 $$
 \langle e^{k+1}-e^{k}, \theta e^{k+1}+(1-\theta)e^k\rangle_\cH=\frac 1 2 ||e^{k+1}||_\cH^2-\frac 1 2  ||e^{k}||_\cH^2
 +(\theta-\frac 1 2) ||e^{k+1}-e^k||_\cH^2
 $$ 
 in \eqref{eq:stab1} we get 
 $$
 \frac 1 2 ||e^{k+1}||_\cH^2-\frac 1 2  ||e^{k}||_\cH^2
 +(\theta-\frac 1 2) ||e^{k+1}-e^k||_\cH^2+\Delta t\cA(\bar{e}_k,\bar{e}_k)=\Delta t \langle \varepsilon^k\,,\,w\rangle_\cH$$
 and then
 \begin{equation}
 \label{eq:stab2}
 \frac 1 2 ||e^{k+1}||_\cH^2-\frac 1 2  ||e^{k}||_\cH^2+
 \Delta t\cA(\bar{e}_k,\bar{e}_k)\leq\Delta t \langle \varepsilon^k\,,\,w\rangle_\cH.
 \end{equation}
But for any $f\in\cV'$ and any $w\in\cV$ we have, using Young's inequality and the coercivity of $\cA$,
$$
\langle f,w\rangle_\cH\leq ||f||_{\cV'}||w||_\cV\leq \frac c 2||w||_\cV^2+\frac{1}{2c}||f||_{\cV'}^2\leq \frac 1 2\cA(w,w)+\frac{1}{2c}||f||_{\cV'}^2.$$
Using this in \eqref{eq:stab2} we get
$$
\frac 1 2 ||e^{k+1}||_\cH^2-\frac 1 2  ||e^{k}||_\cH^2+
 \Delta t\cA(\bar{e}_k,\bar{e}_k)\leq\Delta t \Big(\frac 1 2 \cA(\bar{e}^k,\bar{e}^k)+\frac{1}{2c}||\varepsilon^k||_{\cV'}^2\Big)
$$
and then
$
 ||e^{k+1}||_\cH^2-  ||e^{k}||_\cH^2\leq \frac{\Delta t}{c}||\varepsilon^k||_{\cV'}^2
$.
Summing over $0\leq k\leq m-1$ for any $1\leq m\leq M$ we get \eqref{eq:stabilite}.

\vspace{0.2cm}
\noindent
{\bf Consistency.} Here we follow the lines of the proof of Lemma 7.5-1 in \cite{raviart}. Using \eqref{eq:sol-semifaible-para} and the fact that $v\in C^1(0,T;\cV)$ we can rewrite \eqref{eq:def-eps-m} into: $\forall \, w\in\cV$,
$$
\langle \varepsilon^m\,,\,w\rangle_\cH=
\big\langle  \theta \partial_t v(t_{m+1})+(1-\theta)\partial_t v(t_m)   \big\rangle_\cH
-\frac{1}{\Delta t} \big\langle \Pi v(t_{m+1})-\Pi v(t_m),w\big\rangle_\cH
$$
and further: $\forall w\in\cV$,
$$
\begin{array}{lll}
\langle \varepsilon^m\,,\,w\rangle_\cH&=&
\big\langle  \theta \partial_t v(t_{m+1})+(1-\theta)\partial_t v(t_m)   \big\rangle_\cH
-\dfrac{1}{\Delta t} \big\langle  v(t_{m+1})- v(t_m),w\big\rangle_\cH\\
\\
&&\ds \hspace{3cm}+\dfrac{1}{\Delta t}\int_{t_m}^{t_{m+1}} \big\langle  (I-\Pi)\partial_tv(s) \,,\,w \big\rangle_\cH  ds.
\end{array}
$$
From this we get that for any $0\leq m\leq M-1$
$$
||\varepsilon^m||_{\cV'}=||\varepsilon^m||_{\cH}\leq ||\eta(t_m)||_\cH+\frac{1}{\Delta t}\int_{t_m}^{t_{m+1}} \big|\big|  (I-\Pi)\partial_tv(s) \big|\big|_\cH  ds
$$
with
$$
\eta(t_m)=\frac{1}{\Delta t} \big( v(t_{m+1})- v(t_m)\big)-\theta \partial_t v(t_{m+1})-(1-\theta)\partial_t v(t_m).
$$
We take now $\theta=\frac 1 2$. Using a Taylor expansion with rest in integral form and the fact that 
$v\in C^3(0,T;\cH)$ we get:

\medskip
$$\begin{array}{rcl}
\eta(t_m) & 
 = &\ds \frac{1}{\Delta t}\Big(  \partial_tv(t_m)\Delta t+\frac 1 2 \partial^2_{t^2}v(t_m)(\Delta t)^2+
\frac 1 2 \int_{t_m}^{t_{m+1}}\partial^3_{t^3}v(s)(t_{m+1}-s)^2ds  \Big)\\\\
 & &\;\;\;\;\;\;\;\;\;\;\;\;\;\;\;\;\;\;\;\;\;\;\;\;\;\;\;\;\;\;\;\;\;\;\;\;\;\;\;\;\;\;\;\;\;\;\;\;\;\;\;\;\;\;\;\;\;\;\;\;\;\;\;\;\;\;\;\;\;\displaystyle -\frac 1 2 \partial_t v(t_{m+1})-\frac 1 2 \partial_t v(t_m)\\
&=&\ds \frac 1 2 \partial^2_{t^2}v(t_m)\,\Delta t+
\frac {1}{ 2\Delta t} \int_{t_m}^{t_{m+1}}\partial^3_{t^3}v(s)(t_{m+1}-s)^2ds\\
& &\;\;\;\;\;\;\;\;\;\;\;\;\;\;\;\;\;\;\;\;\;\;\;\;\;\;\;\;\;\;\;\;\;\;\;\;\;\;\;\displaystyle -\frac 1 2 \Big( \partial^2_{t^2}v(t_m)\,\Delta t + \int_{t_m}^{t_{m+1}}\partial^3_{t^3}v(s)(t_{m+1}-s)ds\Big)\\\\
& = &\ds \int_{t_m}^{t_{m+1}}\partial^3_{t^3}v(s)(t_{m+1}-s) \big[-\frac 1 2 + \frac {1}{ 2\Delta t}(t_{m+1}-s)\big]ds\\\\
&=& \ds \frac {1}{ 2\Delta t} \int_{t_m}^{t_{m+1}}  (t_{m+1}-s)(t_n-s)  \partial^3_{t^3}v(s)\,ds.
\end{array}$$

Then we have $||\eta(t_m)||_\cH\leq C\Delta t\int_{t_m}^{t_{m+1}}\big|\big|\partial^3_{t^3}v(s)    \big|\big| _\cH ds$ where $C$ is some universal constant. To sum up we have for any $0\leq m\leq M-1$
\begin{equation*}
||\varepsilon^m||_{\cV'}\leq \frac{1}{\Delta t}\int_{t_m}^{t_{m+1}} \big|\big|  (I-\Pi)\partial_tv(s) \big|\big|_\cH  ds
+C\Delta t\int_{t_m}^{t_{m+1}}\big|\big|\partial^3_{t^3}v(s)    \big|\big| _\cH ds
\end{equation*}
and, using Cauchy-Schwarz inequality,
\begin{equation}
\label{eq:consist}
||\varepsilon^m||_{\cV'}\leq \frac{1}{\sqrt{\Delta t}}\Big(\int_{t_m}^{t_{m+1}} \big|\big|  (I-\Pi)\partial_tv(s) \big|\big|_\cH^2  ds\Big)^{\frac 1 2}
+C(\Delta t)^{\frac 3 2}\Big(\int_{t_m}^{t_{m+1}}\big|\big|\partial^3_{t^3}v(s)    \big|\big| _\cH^2 ds\Big)^{\frac 1 2}.
\end{equation}

\medskip

{\bf Convergence.}
 Let $0\leq m\leq M$. Using now $||v^{m,N,P}-v(t_m)||_\cH\leq ||e^m||_\cH+||v(t_m)-\Pi v(t_m)||_\cH$ and \eqref{eq:stabilite} and \eqref{eq:consist} we get 

\medskip

\noindent $\ds ||v^{m,N,P}-v(t_m)||_\cH\leq\ds ||v(t_m)-\Pi v(t_m)||_\cH+\Big\{  ||f^{N,P}-\Pi f||_\cH^2$

\medskip

$\ds \quad \quad \;\;\;\;\; + \frac 1 c \int_0^{t_m} \big|\big|  (I-\Pi)\partial_tv(s) \big|\big|_\cH^2  ds+\frac C c(\Delta t)^4\int_0^{t_m}\big|\big|\partial^3_{t^3}v(s)    \big|\big| _\cH^2 ds$ 

\medskip

$\ds \quad\quad+ \frac C c(\Delta t)^2  \ds\sum_{k=1}^{m-1}\sum_{j=1}^{m-1}\big( \int_{t_k}^{t_{k+1}} \big|\big|  (I-\Pi)\partial_tv(s) \big|\big|_\cH^2  ds
\times \int_{t_j}^{t_{j+1}} \big|\big|\partial^3_{t^3}v(s)    \big|\big| _\cH^2 ds 
  \big)^{\frac 1 2}
\Big\}^{\frac 1 2}$

\medskip

$\leq \ds \sup_{s\in[0,T]}||(I-\Pi)v(s)||_\cH+ \Big\{ \big(||f^{N,P}-f||_\cH+||f-\Pi f||_\cH\big)^2$

\medskip

$\ds\quad \quad \quad+ \frac T c \sup_{s\in[0,T]}||(I-\Pi)\partial_t v(s)||_\cH^2+\frac C c ||\partial^3_{t^3}v||_{L^2(0,T;\cH)}^2\,(\Delta t)^4$

\begin{equation}\label{eq:conv-para-1}
 \quad+\frac C c T^2\sup_{s\in[0,T]}||\partial^3_{t^3}v(s)||_{\cH}\,\Delta t\, \sup_{s\in[0,T]}||(I-\Pi)\partial_t v(s)||_\cH \Big\}^{\frac 1 2}
\end{equation}
Recall now that we have $||f^{N,P}-f||_\cH\to 0$ as $N,P\to\infty$ and \eqref{eq:conv-proj}. If $y\in C(0,T;\cV)$ the family 
$(I-\Pi)y(s)$, $s\in[0,T]$, is equicontinuous, and we know thanks to \eqref{eq:conv-proj} that for any $s\in[0,T]$ we have
$||(I-\Pi)y(s)||_\cH \to 0$ as $N,P\to\infty$. With the help of Ascoli theorem we can see that we have then
$$
\sup_{s\in[0,T]}||(I-\Pi)y(s)||_\cH\xrightarrow[N\to\infty,P\to\infty]{}0.
$$
Using this with $y=v,\;\partial_tv$, in \eqref{eq:conv-para-1} we get the announced convergence (note that it is in order $(\Delta t)^2$ is time;
cf Remark \ref{rem:crank}).
\end{proof}

\bibliographystyle{spbasic}
\bibliography{BIB_UQ_EDS1-bis.bib}

\end{document}